\documentclass[a4paper,11pt]{article}
\usepackage[T1]{fontenc}
\usepackage[utf8]{inputenc}
\usepackage[english]{babel}
\usepackage{lmodern}
\usepackage{amsmath, amsthm, amssymb} 
\usepackage{authblk}
\usepackage{mathtools} 
\usepackage{enumitem}
\usepackage{fullpage}
\usepackage{bm}
\usepackage{xcolor}
\usepackage{caption}
\usepackage{microtype} 
\usepackage[colorlinks=true]{hyperref}

\microtypesetup{babel=true} 
\renewcommand{\qedsymbol}{$\blacksquare$} 
\setenumerate[0]{label=(\alph*)} 
\hypersetup{ 
    colorlinks,
    linkcolor={red!50!black},
    citecolor={blue!50!black},
	urlcolor={blue!80!black}
}

\makeatletter
\let\save@mathaccent\mathaccent
\newcommand*\if@single[3]{
	\setbox0\hbox{${\mathaccent"0362{#1}}^H$}%
	\setbox2\hbox{${\mathaccent"0362{\kern0pt#1}}^H$}%
	\ifdim\ht0=\ht2 #3\else #2\fi }
\newcommand*\rel@kern[1]{\kern#1\dimexpr\macc@kerna}
\newcommand*\widebar[1]{\@ifnextchar^{{\wide@bar{#1}{0}}}{\wide@bar{#1}{1}}}
\newcommand*\wide@bar[2]{\if@single{#1}{\wide@bar@{#1}{#2}{1}}{\wide@bar@{#1}{#2}{2}}}
\newcommand*\wide@bar@[3]{
	\begingroup
	\def\mathaccent##1##2{
		\let\mathaccent\save@mathaccent
		\if#32 \let\macc@nucleus\first@char \fi
		\setbox\z@\hbox{$\macc@style{\macc@nucleus}_{}$}
		\setbox\tw@\hbox{$\macc@style{\macc@nucleus}{}_{}$}
		\dimen@\wd\tw@ \advance\dimen@-\wd\z@ \divide\dimen@ 3 \@tempdima\wd\tw@ \advance\@tempdima-\scriptspace \divide\@tempdima 10 \advance\dimen@-\@tempdima \ifdim\dimen@>\z@ \dimen@0pt \fi \rel@kern{0.6}\kern-\dimen@
		\if#31 \overline{\rel@kern{-0.6}\kern\dimen@\macc@nucleus\rel@kern{0.4}\kern\dimen@} \advance\dimen@0.4\dimexpr\macc@kerna \let\final@kern#2 \ifdim\dimen@<\z@ \let\final@kern1 \fi
		\if \final@kern1 \kern-\dimen@ \fi
		\else \overline{\rel@kern{-0.6}\kern\dimen@#1} \fi }
	\macc@depth\@ne	\let\math@bgroup\@empty \let\math@egroup\macc@set@skewchar 	\mathsurround\z@ \frozen@everymath{\mathgroup\macc@group\relax} 	 \macc@set@skewchar\relax \let\mathaccentV\macc@nested@a	\if#31 \macc@nested@a\relax111{#1} \else \def\gobble@till@marker##1\endmarker{} \futurelet\first@char\gobble@till@marker#1\endmarker \ifcat\noexpand\first@char A\else \def\first@char{} \fi \macc@nested@a\relax111{\first@char} \fi
	\endgroup }
\makeatother

\newcommand\tstrut{\rule{0pt}{2.9ex}}

\makeatletter
\def\blfootnote{\xdef\@thefnmark{}\@footnotetext}
\makeatother

\DeclareMathOperator{\Exp}{\mathbb{E}}
\DeclareMathOperator{\Var}{\mathbb{V}\mathrm{ar}}
\DeclareMathOperator{\Ind}{\mathbb{I}}

\let\Pr\relax\DeclareMathOperator{\Pr}{\bm{P}} 

\newcommand{\N}{\ensuremath{\mathbb{N}}} 
\newcommand{\Z}{\ensuremath{\mathbb{Z}}} 

\newcommand{\calH}{\ensuremath{\mathcal{H}}} 
\newcommand{\calN}{\ensuremath{\mathcal{N}}} 

\newcommand{\frakp}{\ensuremath{\mathfrak{p}}} 
\newcommand{\frakq}{\ensuremath{\mathfrak{q}}} 
\newcommand{\frakP}{\ensuremath{\mathfrak{P}}} 
\newcommand{\frakS}{\ensuremath{\mathfrak{S}}}

\DeclareMathOperator{\rk}{rk} 
\DeclareMathOperator{\id}{id} 
\DeclareMathOperator{\dom}{dom} 
\DeclareMathOperator{\im}{im} 

\DeclarePairedDelimiter{\floor}{\lfloor}{\rfloor}

\theoremstyle{plain}
\newtheorem{thm}{Theorem}
\newtheorem*{claim}{Claim}

\newtheorem{lemma}[thm]{Lemma}

\theoremstyle{remark}

\theoremstyle{definition}
\newtheorem{defn}[thm]{Definition}

\title{Normal limiting distributions \\for systems of linear equations in random sets}
\author[1]{Juanjo Ru\'e\thanks{Email: juan.jose.rue@upc.edu. Research of J.~R.~supported by the Spanish State Research Agency through projects MTM2017-82166-P, PID2020-113082GB-I00,
the Severo Ochoa and María de Maeztu Program for Centers and Units of Excellence in R\&D (CEX2020-001084-M), and the Marie Curie RISE research network 'RandNet' MSCA-RISE-2020-101007705.}}
\author[2]{Maximilian W\"otzel\thanks{Email: maximilian.woetzel@ru.nl. This research was conducted while M.~W. was a member of the Barcelona Graduate School of Mathematics (BGSMath) as well as the Universitat Polit\`ecnica de Catalunya.
Maximilian W\"otzel acknowledges financial support from the Fondo Social Europeo and the Agencia Estatal de Investigaci\'on through the FPI grant number MDM-2014-0445-16-2 and the Spanish Ministry of Economy and Competitiveness, through the Mar\'ia de Maeztu Programme for Units of Excellence in R\&D (MDM-2014-0445), through the project MTM2017-82166-P, as well as from the Dutch Science Council (NWO) through the grant number OCENW.M20.009.}}
\affil[1]{Departament de Matemàtiques and Institut de Matemàtiques de la UPC (IMTech), Universitat Politècnica de Catalunya \& Centre de Recerca Matemàtica \& Barcelona Graduate School of Mathematics (BGSMath), Barcelona, Spain.}
\affil[2]{Department of Mathematics, Radboud University, Nijmegen, The Netherlands.}
\date{}

\begin{document}
\maketitle

\begin{abstract} 
We consider the binomial random set model $[n]_p$ where each element in $\{1,\dots,n\}$ is chosen independently with probability $p:=p(n)$.
We show that for essentially all regimes of $p$ and very general conditions for a matrix $A$ and a column vector $\bm{b}$, the count of specific integer solutions to the system of linear equations $A\bm{x} = \bm{b}$ with the entries of $\bm{x}$ in $[n]_p$ follows a (conveniently rescaled) normal limiting distribution.
This applies among others to the number of solutions with every variable having a different value, as well as to a broader class of so-called non-trivial solutions in homogeneous strictly balanced systems. 
Our proof relies on the delicate linear algebraic study both of the subjacent matrices and the corresponding ranks of certain submatrices, together with the application of the method of moments in probability theory.
\end{abstract}

\section{Introduction and main results}\label{sec:introMains}

The study of the existence of solutions to linear equations in subsets of the integers (and more generally in additive or even non-abelian groups) has been a prominent topic not only in number theory, but also in extremal combinatorics, ergodic theory, functional analysis and theoretical computers science, among other research areas.
A prototypical example of such investigations is Roth's Theorem~\cite{Roth53}, which proves (by using Fourier analytic means) that dense sets of integers always contain arithmetic progressions of length 3. 
This result was fully generalized by Szemer\'edi~\cite{Szemeredi75}, who obtained a similar statement for arithmetic progressions of all (fixed) lengths.

Following this line of research, Frankl, Graham and R\"odl~\cite[Theorem 2]{FGR88} proved similar results for homogeneous systems of linear equations. 
In particular, they proved that when $A$ is a \emph{density regular} matrix with integer entries (that is, the columns of $A$ sum up to the zero vector), then any dense integer set $T$ will contain solutions to the linear system $A\bm{x}=0$.
This result generalizes Szemer\'edi's Theorem, as arithmetic progressions of length $k$ (or $k$-APs for short) can be encoded as solutions to the linear system $x_2-x_1=x_3-x_2=\dots=x_k-x_{k-1}$. 
Other classical equations, such as the Schur equation for sum-free sets (sets without solutions to the equation $x+y=z$), and Sidon sets (sets without non-trivial solutions to the equation $x+y=z+t$) also fit into this framework.

A very recent trend of investigation is to transfer these extremal results to a sparse setting, and hence to prove analogues on subsets which are dense relative to certain ambient sets.
Specific examples for such sparse ambient sets are for instance the primes or perfect powers, but also the binomial random set $[n]_p$ obtained by choosing to include independently each element in $[n]=\{1,2,\dots,n\}$ with probability $p$.
Returning to the example of $k$-APs, define a set $T$ to be \emph{$(\delta, k)$--Szemer\'edi} if every subset $U\subseteq T$ with at least $\delta|T|$ elements contains a $k$-AP. 
Thus, Szemer\'edi's Theorem shows that $[n]$ (in fact, every set $T$ that is itself dense in $[n]$) is $(\delta, k)$--Szemer\'edi for every $\delta$ and $k$. 
For the specific case of $k=3$, Kohayakawa, {\L}uczak and R\"odl~\cite{KoLuRo96} proved that for all $\delta>0$ there exists constants $c,C>0$ such that if $p\leq c n^{-1/2}$, then as $n$ tends to infinity, $\Pr(\text{$[n]_p$ is $(\delta,3)$--Szemer\'edi})$ tends to $0$, while if $p\geq C n^{-1/2}$, it tends to $1$. 
The existence of this threshold was extended to all values of $k$ in~\cite{Scha16,CoGo16}, and later rediscovered in~\cite{SaTh15,BaMoSa15} in the context of independent sets in hypergraphs.
These ideas, named \emph{hypergraph container method} were then used by Spiegel~\cite{Spiegel2017} and independently by Hancock, Staden and Treglown~\cite{HaStTr2019} to extend~\cite[Theorem 2]{FGR88} to the broadest class of linear systems possible.
Similar techniques were used by Ru\'e, Serra and Vena~\cite{RuSeVe17} to study random sparse analogues of~\cite[Theorem 2]{FGR88} in finite fields and more general configurations than linear systems of equations.

The main contribution of our paper follows this trend of research, with the aim of studying the \emph{typical} number of solutions to linear systems of equations in random sets of integers. 
Our main theorem, Theorem~\ref{thm:distribution} is technical and in order to formally state it we need to introduce a wide variety of algebraic definitions and lemmas, which will be the core of Section~\ref{sec:prelimLems}. 
However, we will now introduce some main definitions in order to be able to formulate two important consequences of Theorem~\ref{thm:distribution}, namely Theorems~\ref{thm:distProper} and~\ref{thm:distHomStrictBal}.

\paragraph{Proper solutions and their distribution}
In order to properly state our results we need to introduce some notation.
Let $m>r$ be positive integers, $A\in\Z^{r\times m}$ an integer matrix, and $\bm{b}\in\Z^{r}$ an integer vector.
We write the \emph{rank} of a matrix $A$ by $\rk(A)$.
Define $S(A,\bm{b})=\{\bm{x}\in\Z^m : A\bm{x}=\bm{b}\}$ as the set of integer solutions to the system of linear equations $A\bm{x}=\bm{b}$.
Furthermore, if $Q\subset [m]$, let $A^Q\in\Z^{r\times |Q|}$ denote the $r\times |Q|$ matrix obtained by only keeping the columns of $A$ indexed by $Q$.
We notice that the rank of the empty matrix $A^\emptyset$ is zero.
We will identify tuples $\bm{x}=(x_1,\dots,x_m)\in\Z^m$ with the corresponding column vectors but abuse notation slightly by letting $\bm{x}^Q$ denote the vector obtained by only keeping the \emph{rows} of $\bm{x}$.
This should not be confused with the notation $\bm{x}^k$, where $k$ is a positive integer, which we will use to denote the vector \[\bm{x}^k = (\underbrace{\bm{x},\dots,\bm{x}}_\text{$k$ times}) = (x_1,\dots,x_m,x_1,\dots,x_m,\dots,x_1,\dots,x_m)\in \Z^{km}.\]
Finally, if $\bm{x}=(x_1,\dots,x_m)\in\Z^m$ is an $m$-tuple, we will write $\{\bm{x}\}$ as a shorthand for the set $\{x_1,\dots,x_m\}$.
In particular, if some entry in $\bm{x}$ is repeated, then the cardinality of $\{\mathbf{x}\}$ is strictly less than $m$.

Roughly speaking, our main goal in this paper is to estimate $|S(A,\bm{b})\cap [n]_p^m|$, that is, the number of solutions whose entries lie in $[n]_p$ for different regimes of $p$.
In order for this to be a well-posed problem, we need some definitions from~\cite{Spiegel2020} and~\cite{RSZ2018}.

\begin{defn}\label{defn:positive_abundant}
An integer matrix $A\in\Z^{r\times m}$ is said to be
\begin{enumerate}[label=\roman*)]
\item \emph{positive} if there exists an integer solution to $A\bm{x}=\bm{0}$ having all positive entries: \[S(A,\bm{0})\cap \N^m \neq \emptyset,\]
and if for any pair of indices $i,j\in[m]$, $i\neq j$, there exists a solution $(x_1,\dots,x_m)\in S(A,\bm{0})$ satisfying $x_i\neq x_j$.
\item \emph{abundant} if $\rk(A)>0$ and the removal of at most two columns does not change the rank of $A$: \[\rk(A^Q)=\rk(A)\] for any $Q\subset[m]$ satisfying $|Q|\geq m-2$.
\end{enumerate}
\end{defn}

Sometimes the second requirement for positivity in Definition~\ref{defn:positive_abundant} is called \emph{irredundancy}.
For our applications, we will never consider positivity and irredundancy separately, and hence we have combined those two properties into a single one for expediency's sake.

In general, considering any possible solution in $S(A,\bm{b})$ is a bit too lenient.
For instance, in the $3$-AP situation (whose associated matrix is $A=\begin{pmatrix}1 &-2 &1\end{pmatrix}$), $S(A,\bm{0})$ also contains all the tuples $(a,a,a)$ with $a\in\Z$, and those are clearly not of interest.
To remedy this issue we need to introduce the easy notion of \emph{proper} solutions to $A\bm{x}=\bm{b}$.

\begin{defn}\label{defn:proper}
Let $m>r$ be positive integers, $A\in\Z^{r\times m}$, and $\bm{b}\in\Z^r$.
Then the set $S_0(A,\bm{b})$ of \emph{proper} solutions is the subset of $S(A,\bm{b})$ with all coordinates being pairwise distinct, that is \[S_0(A,\bm{b}) = \{(x_1,\dots,x_m) \in S(A,\bm{b}) : x_i\neq x_j \text{ for all }1 \leq i < j\leq m\}.\]
\end{defn}

Before proceeding to the statement of our first main result, we need to define one more parameter, which will measure the densest subsystem of $A$.
The motivation behind this can be compared to the graph setting, where one wants to study the number of occurrences of some fixed graph $G$ as a subgraph of a binomial random graph with $n$ vertices and edge probability $p$.
Here, one naively would expect the threshold for a random graph to contain $G$ to be when $p$ is around $n^{-v(G)/e(G)}$, which is when the expected number of copies of $G$ flips from $0$ to positive.
But this is not the case: if $G$ contains a subgraph $H$ with $v(H)/e(H)>v(G)/e(G)$, then $n^{-v(H)/e(H)}$ will define the threshold instead.
A similar behavior occurs for the distribution of subgraph counts, as shown by Ruci\'{n}ski in~\cite{Rucinski1988}.

\begin{defn}\label{defn:systemdensity}
For positive integers $m>r$ and a positive integer matrix $A\in\Z^{r\times m}$, define the \emph{density} $c(A)$ of $A$ by \[c(A) = \max_{\emptyset\neq Q\subset[m]} \frac{|Q|}{|Q|-r_Q},\]
where $r_Q = r_Q(A)= \rk(A) - \rk(A^{\widebar{Q}})$, and $\widebar{Q}=[m]\setminus Q$.
\end{defn}

Note that this is indeed well-defined, since for a positive matrix we see that $\rk(A^{\widebar{Q}})\geq \rk(A)-|Q|+1$ for every $\emptyset\neq Q\subset m$ (see, for instance, the proof of Lemma~\ref{lemma:sols_with_fixed_elts}).

We are now ready to state our first main theorem.
For a random variable $X$ with finite first moment $\mathbb{E}(X)$ and non-zero finite variance $\Var(X)$, denote by $\tilde{X}=(X-\Exp(X))/\sqrt{\Var(X)}$ its normalization.
We write $X_n \xrightarrow{d} Y$  when a sequence of random variables $\{X_n\}_{n\geq 1}$ tends in distribution to $Y$.

\begin{thm}\label{thm:distProper}
Let $m>r$ be positive integers, $A\in\Z^{r\times m}$ a positive and abundant integer matrix, and $\bm{b}\in\Z^r$ such that $S(A,\bm{b})\neq\emptyset$.
Furthermore, let $n$ be an integer, $0\leq p:=p(n)\leq 1$ and $X_n$ the random variable equal $|S_0(A,\bm{b})\cap[n]_p^m|$, which counts the number of proper solutions $\bm{x}\in[n]_p^m$ to $A\bm{x}=\bm{b}$.
Then \[\tilde{X}_n \xrightarrow{d} \calN(0,1)\] if $n(1-p)\to\infty$ and $np^{c(A)}\to\infty$.
\end{thm}

Note that in the specific case of $k$-APs, this problem was already investigated by Barhoumi-Andr\'eani, Koch and Liu~\cite{BaAnKoLi2019}.
In fact, they prove not only results on the limiting distribution of the number of $k$-term arithmetic progressions in $[n]_p$ even when $k$ is unbounded (but sub-logarithmic), they also establish a bivariate central limit theorem for the joint distribution when considering the counts of two distinct progression lengths.
In contrast, Theorem~\ref{thm:distProper} in the setting of progressions requires the length to be fixed.

\paragraph{Non-trivial solutions, and main theorem for their distribution}
Proper solutions are always of interest in the number theoretical context, but in a wide variety of situations we need to take care of non-proper solutions.
For example, for $A=\begin{pmatrix}1 &1 &-1 &-1\end{pmatrix}$, we see that $S(A,\bm{0})$ is the number of additive quadruples satisfying $a+b=c+d$.
Clearly, if $\{a,b\}\cap\{c,d\}\neq\emptyset$, then these two sets must in fact be the same.
On the other hand, solutions of the form $2a=c+d$ with $a\notin\{c,d\}$ are not proper but obviously of interest and might be considered as valid ones.

Keeping this in mind, we will now recall the notion of \emph{non-trivial} solutions for systems of linear equations due to Ru\'e, Spiegel and Zumalac\'{a}rregui~\cite{RSZ2018} which generalized an earlier notion of this for single line equations introduced by Ruzsa~\cite{Ruzsa1993}.
For a solution $\bm{x}=(x_1,\dots,x_m)\in S(A,\bm{b})$, we define $\frakp(\bm{x}) \subset 2^{[m]}$ to be the partition of $[m]$ such that for any $i,j\in[m]$ it holds that $x_i=x_j$ if and only if $i$ and $j$ are in the same partition class of $\frakp(\bm{x})$.
One can view $\frakp(\bm{x})$ as an ordered $|\frakp(\bm{x})|$-tuple $(C_1,\dots,C_{|\frakp(\bm{x})|})$ such that $\min C_i < \min C_j$ whenever $i<j$.
Doing this, we can now define the matrix $A_{\frakp(\bm{x})}$ in the following way.
Suppose the columns of $A$ are denoted by $\bm{c}_1,\dots,\bm{c}_m\in\Z^r$, then
\begin{equation*}
A_{\frakp(\bm{x})} = {\begin{pmatrix} \sum_{i\in C_1}\bm{c}_i &\big\vert &\sum_{i\in C_2}\bm{c}_i &\big\vert &\cdots &\big\vert &\sum_{i\in C_{|\frakp(\bm{x})|}}\bm{c}_i\end{pmatrix}},
\end{equation*}
so $A_{\frakp(\bm{x})}\in\Z^{r\times |\frakp(\bm{x})|}$ is the $r\times|\frakp(\bm{x})|$ matrix obtained by contracting all columns in the same partition class, with columns ordered by the minimum index in each class.
Finally, define the set
\begin{equation*}
\frakP(A) = \{\frakp \subset 2^{[m]} : \frakp \text{ is a partition of }[m]\text{ and }\rk(A_\frakp) = \rk(A)\}.
\end{equation*}

We are finally ready to introduce the notion of a \emph{non-trivial} solution.

\begin{defn}\label{defn:nontrivial}
Let $m>r$ be positive integers, $A\in\Z^{r\times m}$, and $\bm{b}\in\Z^r$.
Then the set $S_1(A,\bm{b})$ of \emph{non-trivial} solutions is the subset of $S(A,\bm{b})$ with associated partitions coming from $\frakP(A)$, that is \[S_1(A,\bm{b}) = \{\bm{x}\in S(A,\bm{b}) : \frakp(\bm{x})\in\frakP(A)\}.\]
\end{defn}

This definition might seem quite arbitrary, but the interested reader is invited to read the discussions in~\cite{RSZ2018} and~\cite{Spiegel2020} which show that, in some sense, it is quite natural in that it encompasses the natural notions of non-triviality for specific systems of linear equations studied in the literature.
When we want to investigate the distribution of nontrivial solutions, we actually need to look at $c(A_\frakp)$ for any partition type $\frakp$ that is to be considered.
A special case in which it suffices to only consider $c(A)$ is that of \emph{strictly balanced} systems.

\begin{defn}\label{defn:strictlyBalanced}
Let $m>r$ be positive integers and $A\in\Z^{r\times m}$ be positive.
Then the matrix $A$ is called \emph{strictly balanced} if \[c(A)=\frac{m}{m-\rk(A)} > \max_{\emptyset\subsetneqq Q \subsetneqq [m]}\frac{|Q|}{|Q|-r_Q(A)},\] and furthermore for every $\frakp\in\frakP(A)\setminus\{\{1\},\{2\},\dots,\{m\}\}$ such that $A_\frakp$ is positive, it holds that $c(A)>c(A_\frakp)$.
\end{defn}

We can prove the following theorem about the distribution of nontrivial solutions of strictly balanced systems of linear equations.

\begin{thm}\label{thm:distHomStrictBal}
Let $m>r$ be positive integers, $A\in\Z^{r\times m}$ a positive and abundant strictly balanced integer matrix such that $S(A,\bm{0})\neq\emptyset$.
Furthermore, let $n$ be an integer, $0\leq p:=p(n)\leq 1$ and $X_n$ the random variable $|S_1(A,\bm{0})\cap[n]_p^m|$ that counts the number of non-trivial solutions $\bm{x}\in[n]_p^m$ to $A\bm{x}=\bm{0}$.
Then \[\tilde{X}_n \xrightarrow{d} \calN(0,1)\] if $n(1-p)\to\infty$ and $np^{c(A)}\to\infty$.
\end{thm}

Both of our main results can be compared to those proved in~\cite{RSZ2018} and~\cite{Spiegel2020}. In the former the authors investigated the threshold behavior of $|S_1(A,\bm{0})\cap[n]_p^m|$ when $A\in\Z^{r\times m}$ is a positive matrix and in the latter, Spiegel extended this to $|S_1(A,\bm{b})\cap[n]_p^m|$ for an arbitrary $\bm{b}\in\Z^r$ when $A$ is positive and also abundant.

In that sense, their results hold in a more general setting, even when compared to our main technical result, Theorem~\ref{thm:distribution}.
However, we note that the study of the distribution is more delicate than the threshold behavior, as can be seen by another result from~\cite{RSZ2018}, in which the authors explore the behavior at the threshold of nontrivial solutions to $A\bm{x}=\bm{0}$ for strictly balanced systems, that is, they work exactly in the same setting as in Theorem~\ref{thm:distHomStrictBal}.

To conclude this section, let us mention that the core of the proof is based on a combination of algebraic ideas dealing with the matrix associated with the linear system of equations and the method of moments in probability theory to show convergence towards a normal distribution. 
The second part is highly inspired by the techniques developed by Ruci\'{n}ski in~\cite{Rucinski1988} where he proved the following result on the distribution of the number of occurrences of a fixed graph $G$ as a subgraph in a binomial random graph.

\begin{thm}[\cite{Rucinski1988}]\label{thm:rucinski}
Let $G$ be a graph on $v(G)$ vertices and $e(G)$ edges.
Furthermore, let $n$ be an integer, $0\leq p := p(n)\leq 1$ and $X_n$ the random variable that the number of subgraphs of $G(n,p)$ that are isomorphic to $G$.
Then \[\tilde{X}_n \xrightarrow{d} \calN(0,1)\] if and only if $n^2(1-p)\to\infty$ and $np^{d(G)}\to\infty$, where $d(G)=\max_{H\subset G}e(H)/v(H)$ is the \emph{density} of $G$.
\end{thm}

The proof of our main result, Theorem~\ref{thm:distribution} which in particular implies Theorems~\ref{thm:distProper} and~\ref{thm:distHomStrictBal} follows a similar structure as that of Theorem~\ref{thm:rucinski}.
Specifically, in order to show that the stated conditions on $p$ are sufficient, one considers three different ranges of $p$ and determines the exact structure of the objects that actually affect the moments.

It turns out that when viewed through the lens of hypergraphs, these important objects will actually be essentially the same for both graphs and systems of linear equations.
However, we want to stress that actually establishing this fact requires the delicate study of the more difficult linear algebraic structure of the patterns we want to take into account, and hence our work highly differs from~\cite{Rucinski1988}.

\paragraph{Plan of the paper} 
In Section~\ref{sec:prelimLems} we will state some needed prior results from~\cite{RSZ2018} and~\cite{Spiegel2020}, as well as prove intermediate lemmas that are needed to establish a meta result, namely Theorem~\ref{thm:distribution} from which one can deduce Theorems~\ref{thm:distProper} and~\ref{thm:distHomStrictBal}. 
Section~\ref{sec:mainProof} will then be devoted entirely to stating and proving this main meta result, Theorem~\ref{thm:distribution}.
To do this, we will be using the so-called method of moments to analyze three distinct regimes of the probability $p$.
We conclude with a discussion on some remaining open problems in Section~\ref{sec:conclusion}.
In particular, we will investigate the question of whether the sufficient conditions on $p$ in Theorems~\ref{thm:distProper},~\ref{thm:distHomStrictBal} and~\ref{thm:distribution} are also necessary.

\section{Algebraic properties of systems of linear equations}\label{sec:prelimLems}

Let us start by investigating the relation between proper and non-trivial solutions.
For any $A\in\Z^{r\times m}$ and $\bm{b}\in\Z^r$ it is clear that \[S_0(A,\bm{b}) \subset S_1(A,\bm{b}) \subset S(A,\bm{b}).\]
Furthermore, suppose $s>1$ and we fix a specific $(C_1,\dots,C_s)=\frakp\in\frakP(A)$ such that $\min C_i<\min C_j$ whenever $i<j$.
Then if $\bm{x}=(x_1,\dots,x_m)\in S_1(A,\bm{b})$ satisfies $\frakp(\bm{x})=\frakp$, we see that $\tilde{\bm{x}}=(x_{\min C_1},\dots,x_{\min C_s})\in\Z^{s}$ has $s$ pairwise distinct coordinates and $A_\frakp \tilde{\bm{x}} = \bm{b}$, so $\tilde{\bm{x}}\in S_0(A_\frakp,\bm{b})$.
Conversely, for any proper solution $\bm{x}=(x_1,\dots,x_s)$ in $S_0(A_\frakp,\bm{b})$, we see that the vector $\tilde{\bm{x}}\in\Z^m$ with $i$-th coordinate $\tilde{x}_i$ defined as \[\tilde{x}_i = \sum_{j=1}^s x_j\cdot\delta_{i\in C_j}\] satisfies $\frakp(\tilde{\bm{x}}) = \frakp$ and $A\tilde{\bm{x}}=\bm{b}$ (here $\delta_{i\in C_j}$ denotes the indicator function which takes values $0$ or $1$ if $i\not\in C_j$ and $i\in C_j$, respectively).
So we get the following lemma, which is essentially Lemma~1.10 in~\cite{Spiegel2020}.

\begin{lemma}[\cite{Spiegel2020}]\label{lemma:nontrivial_reduced_proper}
Let $m>r$ be positive integers, $A\in\Z^{r\times m}$, and $\bm{b}\in\Z^r$.
Then for any partition $\frakp\in\frakP(A)$ and any subset $T\subset \Z$, it holds that \[|\{\bm{x} \in S_1(A,\bm{b}) : \frakp(\bm{x})=\frakp\}\cap T^m| = |S_0(A_\frakp,\bm{b})\cap T^{|\frakp|}|.\]
\end{lemma}

Lemma~\ref{lemma:nontrivial_reduced_proper} will be very helpful because if $A$ is positive, then the approximate size of $S_0(A,\bm{b})$ is not difficult to determine for any vector $\bm{b}$.
Specifically, we have the following result which is Lemma~1.4 in~\cite{Spiegel2020}.

\begin{lemma}[\cite{Spiegel2020}]\label{lemma:proper_sols_count}
Let $m>r$ be positive integers, $A\in\Z^{r\times m}$, and $\bm{b}\in\Z^r$ such that $S(A,\bm{b})\neq\emptyset$.
Then if $n\in\N$ it holds that \[|S_0(A,\bm{b})\cap[n]^m| \leq |S(A,\bm{b})\cap[n]^m|\leq n^{m-\rk(A)}.\]
Furthermore, if $A$ is positive we also have \[|S_0(A,\bm{b})\cap[n]^m| = \Omega(n^{m-\rk(A)}).\]
\end{lemma}

In~\cite{RSZ2018} the authors managed to find a threshold probability function for when non-trivial solutions to a homogeneous system of linear equations appear in $[n]_p$, but for our current purposes we need to make one more specification regarding types of solutions.
The issue essentially lies in the fact that while $A$ is positive, the same might not hold for $A_\frakp$ for some $\frakp\in\frakP(A)$.
Note that this can only happen in the case $\bm{b}\neq \bm{0}$, since clearly $S_0(A_\frakp,\bm{0})\neq\emptyset$ implies that $A_\frakp$ is positive.
As a problematic example, consider for instance the matrix $A=\begin{pmatrix}1 &1 &1 &1 &-1\end{pmatrix}$ which is positive since for instance $\bm{x}=(1,2,3,4,10)\in \N^5$ is a proper solution to the system $A\bm{x}=\bm{0}$.
But taking $\bm{b}=6$, we see that $\bm{y}=(1,2,3,3,3)$ satisfies $A\bm{y}=\bm{b}$.
We have $\frakp(\bm{y})=(\{1\},\{2\},\{3,4,5\})$ and hence $A_{\frakp(\bm{y})} = \begin{pmatrix}1 &1 &1\end{pmatrix}$ which is clearly not positive.

The problem arises as follows: if $A_\frakp$ is not positive but we still have a significant number of solutions of this type -- this can for instance be achieved by using the previous example as a gadget in a larger system -- the number of solutions in the random binomial set will depend heavily on whether or not a specific bounded number of elements are included, and so getting any good results on their distribution is unlikely.

Having stated this motivation, we first introduce the concept of \emph{positive partitions} as the subset $\frakP_0(A)\subset \frakP(A)$ such that
\begin{equation*}\label{eq:positive_partitions}
\frakP_0(A) = \{\frakp\in\frakP(A) : A_\frakp \text{ is positive}\}.
\end{equation*}
For any $\frakP\subset \frakP(A)$ we can now define the concept of \emph{type-$\frakP$} solutions.

\begin{defn}\label{defn:typeP-solutions}
Let $m>r$ be positive integers, $A\in\Z^{r\times m}$, $\bm{b}\in\Z^r$, and $\frakP\subset\frakP(A)$.
Then the set $S_\frakP(A,\bm{b})$ of \emph{type-$\frakP$} solutions is the subset of $S(A,\bm{b})$ with associated partitions coming from $\frakP$, that is
\[S_\frakP(A,\bm{b}) = \{\bm{x}\in S(A,\bm{b}) : \frakp(\bm{x})\in\frakP\}.\]
If $\frakP=\{\frakp\}$, we will write $S_\frakp(A,\bm{b})$ instead of $S_{\{\frakp\}}(A,\bm{b})$.
\end{defn}

Non-trivial solutions are therefore the same as type-$\frakP(A)$ solutions.
The following is a straightforward consequence of Lemma~\ref{lemma:nontrivial_reduced_proper} and the upper bound in Lemma~\ref{lemma:proper_sols_count}.

\begin{lemma}\label{lemma:sols_with_fixed_elts}
Let $m>r$ and $n$ be positive integers, $A\in\Z^{r\times m}$ positive, $\bm{b}\in\Z^r$, $\frakP\subset\frakP_0(A)$, and $Z\subset [n]$ a fixed set.
Then \[\big|\{\bm{y}\in S_\frakP(A,\bm{b})\cap[n]^m : \{\bm{y}\}\cap Z \neq \emptyset\}\big| = O(n^{m-\rk(A)-1}).\]
Furthermore, if $A$ is also abundant, then \[\big|\{\bm{y}\in S_\frakP(A,\bm{b})\cap[n]^m : |\{\bm{y}\}\cap Z|\geq 2\}\big| = O(n^{m-\rk(A)-2}).\]
\end{lemma}

\begin{proof}
We will consider each partition $\frakp\in\frakP$ separately. Due to Lemma~\ref{lemma:nontrivial_reduced_proper}  we can consider $\bm{y}\in S_0(A_\frakp,\bm{b})\cap [n]^{|\frakp|}$.
If $\bm{y}$ contains an element of $Z$, then there exists some non-empty index set $Q\subset[|\frakp|]$ and a vector $\bm{z}\in Z^{|Q|}$ such that $\bm{y}^{\widebar{Q}} \in S_0(A_\frakp^{\widebar{Q}}, \bm{b}-A_\frakp^Q\bm{z})\cap [n]^{|\frakp|-|Q|}$ (note that here, $\widebar{Q}=[|\frakp|]\setminus Q$).
By Lemma~\ref{lemma:proper_sols_count}, we have
 \[\big|S_0(A_\frakp^{\widebar{Q}}, \bm{b}-A_\frakp^Q\bm{z})\cap [n]^{|\frakp|-|Q|}\big| \leq n^{|\frakp|-|Q|-\rk(A_\frakp^{\widebar{Q}})}.\]
$A_\frakp$ is positive, which implies that \[\rk(A_\frakp^{\widebar{Q}}) \geq \rk(A_\frakp)-|Q|+1=\rk(A)-|Q|+1.\]
Indeed, the positivity implies that the removal of any single column will not change the rank of the resulting matrix, and since the rank of a matrix is the dimension of its column space, any subsequent removal decreases the rank by at most $1$.
Hence
\begin{equation}\label{eq:sols_with_fixed_elts_positive}
|\frakp|-|Q|-\rk(A_\frakp^{\widebar{Q}}) \leq |\frakp|-\rk(A)-1 \leq m-\rk(A)-1.
\end{equation}
Since there are only $O(1)$ choices for $\frakp$, $Q$ and $\bm{z}$, this implies the result.
For the second part, note first that for any $|\frakp|<m$, Relation~\eqref{eq:sols_with_fixed_elts_positive} actually implies the statement already, so it only remains to consider the case $|\frakp|=m$, that is, $\bm{y}\in S_0(A,\bm{b})\cap[n]^m$.
If $A$ is abundant, the removal of any two columns keeps the rank constant, while any subsequent removal will decrease it by at most $1$ at a time.
Hence for any $Q\subset[m]$ with $|Q|\geq 2$ we have \[\rk(A^{\widebar{Q}})\geq \rk(A)-|Q|+2,\] and so \[m-|Q|-\rk(A^{\widebar{Q}}) \leq m-\rk(A)-2,\] which is what we wanted to show.
\end{proof}

\subsection{Compounded matrices}\label{ssec:compounde-matrices}

While Lemma \ref{lemma:sols_with_fixed_elts} showed that the upper bound of Lemma~\ref{lemma:proper_sols_count} is always helpful in the situation of counting solutions with some entries fixed beforehand, the requirement of positivity is sometimes too restrictive to do the same for the lower bound.
Suppose $\bm{x}\in S_1(A,\bm{b})$ is some fixed non-trivial solution. We want then to count the number of solutions $\bm{y}\in S_1(A,\bm{b})$ that intersect $\bm{x}$, that is, the size of the set $S_0(A_{\frakp(\bm{y})}^{\widebar{Q}}, \bm{b}-A_{\frakp(\bm{y})}^Q\tilde{\bm{x}})$, where $Q\subset [|\frakp(\bm{y})|]$ is some index set and $\tilde{\bm{x}}\in\{\bm{x}\}^{|Q|}$.
Lemma~\ref{lemma:sols_with_fixed_elts} with $Z=\{\bm{x}\}$ immediately gives helpful upper bounds, but there are two issues when trying to apply the lower bound of Lemma~\ref{lemma:proper_sols_count} directly.

The first is that when summing over several distinct $Q$, the same solution will be counted multiple times, but this could be alleviated by just counting a single $Q$ that maximizes the cardinality of the corresponding set.
The bigger issue is that it is not clear at all that the matrix $A_{\frakp(\bm{y})}^{\widebar{Q}}$ will be positive, and in fact this is not true in general even when $A_{\frakp(\bm{y})}$ is itself positive.
Consider for instance the matrix $A=\begin{pmatrix}1 &1 &-1\end{pmatrix}$ (associated with the \emph{Schur equation} $x+y=z$), which is both positive and abundant.
Then for any $Q$ with $|Q|=1$, the equation $A^{\widebar{Q}}(x,y)=\bm{0}$ implies either $x=y$ or $x=-y$, so in either case positivity will be violated.

Instead, we will consider the concept of the \emph{compounded} matrix already used in~\cite{RSZ2018} to study the distribution at the threshold.
For matrices $A\in\Z^{r_A\times m_A}$, $B\in\Z^{r_B\times m_B}$ and a bijection $M\colon P\to [m_B]$ with $P=\{p_1<\dots<p_{|P|}\}\subset [m_A]$, define the $(r_A+r_B)\times(m_A+m_B-|P|)$ matrix $A\overset{M}{\times}B$ as
\begin{equation*}\label{eq:compoundedMatrix}
A\overset{M}{\times}B=\left(
\begin{array}{c|c|c|c|c}
A^{[m_A]\setminus P} &\bm{a}_{p_1} &\dots &\bm{a}_{p_{|P|}} &\bm{0}\\
\hline
\bm{0} &\bm{b}_{M(p_1)} &\dots &\bm{b}_{M(p_{|P|})} &B^{[m_B]\setminus M(P)} \tstrut
\end{array}
\right),
\end{equation*}
where $\bm{a}_1,\dots,\bm{a}_{m_A}$ denote the columns of $A$ and $\bm{b}_1,\dots,\bm{b}_{m_B}$ the columns of $B$.
Note that while we can apply this operator iteratively, the operation is in general not associative or even well-defined, that is \[(A\overset{M}{\times}B)\overset{M'}{\times}C \neq A\overset{M}{\times}(B\overset{M'}{\times}C).\]
An exception to this is the case when $\dom(M)=\emptyset$, which we will abbreviate by writing $A\overset{\bm{.}}{\times}B$.
In general, whenever no parentheses are used the compounded matrix is implied to be constructed from left to right iteratively, that is $A\overset{M}{\times}B\overset{M'}{\times}C:=(A\overset{M}{\times}B)\overset{M'}{\times}C$. 
An important example that will appear often concerns the case $A=B$ and $M=id_Q$ the identity function for some index set $Q\subset [m_A]$, that is, \[A\overset{\id_Q}{\times}A = \left(\begin{array}{c|r|c}A^{\widebar{Q}} &A^Q &\bm{0}\\ \hline \bm{0} &A^Q &A^{\widebar{Q}}\tstrut\end{array}\right).\]
The reason to study these matrices is natural: Let $\rho_A=\rho_A(M)\colon[m_A]\to[m_A]$ be the bijective function that maps the column indices of the first $[m_A]$ columns of $A\overset{M}{\times}A$ to the corresponding ones of $A$, and define $\rho_B\colon [m_A+m_B-|P|]\setminus [m_A-|P|]\to [m_B]$ similarly.
Recall that for any integer $t>1$, $\bm{b}^t$ denotes the vector $(\bm{b},\dots,\bm{b})\in\Z^{rt}$.
Then $\bm{z}=(z_1,\dots,z_{m_A+m_B-|P|})$ is a proper integer solution to the system of linear equations $(A\overset{M}{\times}B)\bm{z}=\bm{b}^2$ if and only if \[\bm{x}=(x_1,\dots,x_{m_A}):=(z_{\rho^{-1}_A(1)},\dots,z_{\rho^{-1}_A(m_A)}) \text{ and } \bm{y}=(y_1,\dots,y_{m_B}):=(z_{\rho^{-1}_B(1)},\dots,z_{\rho^{-1}_A(m_B)})\] are proper solutions to the systems $A\bm{x}=\bm{b}$ and $B\bm{y}=\bm{b}$, and for any $i\in[m_A]$ and $j\in[m_B]$ it holds that $x_i=y_j$ if and only if $i\in P$ and $j=M(i)$.
In other words, elements of $S_0(A\overset{M}{\times}B,\bm{b}^2)$ correspond to pairs of proper solutions in $S_0(A,\bm{b})$ and $S_0(B,\bm{b})$ that intersect exactly in the coordinates indicated by the function $M$.

We first state an easy but important property that any compounded matrix satisfies:

\begin{lemma}\label{lemma:compMatRankLB}
Let $m_A>r_A$ and $m_B>r_B$ be positive integers, $A\in\Z^{r_A\times m_A}$, $B\in\Z^{r_B\times m_B}$, $P\subset [m_A]$, $M\colon P\to [m_B]$ a bijection and $Q\subset [m_A+m_B-|P|]$.
If we define $Q_1 = Q\cap [m_A]$ and $Q_2=Q\setminus Q_1$, then \[\rk\bigl((A\overset{M}{\times}B)^{\widebar{Q}}\bigr) \geq \rk(A^{[m_A]\setminus Q_1})+\rk(B^{\widebar{M(P)}\setminus Q_2}).\]
In particular, $\rk(A\overset{M}{\times}B) \geq \rk(A)+\rk(B^{\widebar{M(P)}})$.
\end{lemma}

\begin{proof}
It is clear that the rank of $(A\overset{M}{\times}B)^{[m_A]\setminus Q_1}$ is at least $rk(A^{[m_A]\setminus Q_1})=:r_1$, and similarly the rank of $(A\overset{M}{\times}B)^{[m_A+m_b-|P|]\setminus ([m_A]\cup Q_2)}$ is at least $\rk(B^{\widebar{M(P)}\setminus Q_2})=:r_2$, so let \(\bm{c}_1,\dots,\bm{c}_{r_1}\) be a collection of $r_1$ linearly independent vectors from the first $m_A-|Q_1|$ columns of $(A\overset{M}{\times}B)^{\widebar{Q}}$ such that the corresponding column vectors of $A^{[m_A]\setminus Q_1}$ -- that is, the vectors obtained by only considering the first $r_A$ rows -- are also linearly independent, and \(\bm{c}_{r_1+1},\dots,\bm{c}_{r_1+r_2}\) a collection of $r_2$ of independent vectors from the $m_B-|P|-|Q_2|$ last columns.
Then since the entries in the first $r_A$ rows of $\bm{c}_{r_1+1},\dots,\bm{c}_{r_1+r_2}$ are $0$ and the column vectors obtained by only considering the first $r_A$ rows of $\bm{c}_1,\dots,\bm{c}_{r_1}$ are linearly independent, any linear representation of $\bm{0}$ \[\lambda_1 \bm{c}_1 + \dots + \lambda_{r_1+r_2}\bm{c}_{r_1+r_2}\] must satisfy $\lambda_1=\dots=\lambda_{r_1}=0$.
But since the remaining $r_2$ columns were also linearly independent, we must have $\lambda_{r_1+1}=\dots=\lambda_{r_1+r_2}=0$ as well, and hence the $r_1+r_2$ columns $\bm{c}_1,\dots,\bm{c}_{r_1+r_2}$ are linearly independent.
Since the rank of a matrix is the dimension of its column space we are done.
\end{proof}

Next we will show that for certain compounded matrices this is indeed the correct rank.

\begin{lemma}\label{lemma:compMatRankUB}
Let $m>r$ be positive integers, $A\in\Z^{r\times m}$, and $Q\subset [m]$.
Then \[\rk(A\overset{\id_Q}{\times}A) = \rk(A)+\rk(A^{\widebar{Q}}).\]
\end{lemma}

\begin{proof}
We prove by obtaining the same upper and lower bounds for $\rk(A\overset{\id_Q}{\times}A)$. 
The lower bound follows from Lemma~\ref{lemma:compMatRankLB}.
For the upper bound, note that by performing elementary row and column operations we see that \[\rk(A\overset{\id_Q}{\times}A) = \rk\left(\begin{array}{c|r|c}A^Q &A^{\widebar{Q}} &\bm{0}\\ \hline \bm{0} &-A^{\widebar{Q}} &A^{\widebar{Q}}\tstrut\end{array}\right).\]
Let $r_1=\rk(A)$ and $r_2=\rk(A^{\widebar{Q}})$ and suppose we have $r_1+r_2+1$ row vectors from this matrix.
If at least $r_1+1$ rows come from the upper half, then there exists a nontrivial representation of $\bm{0}\in\Z^{1\times (2m-|Q|)}$, and hence the remaining coefficients can be set to $0$.
If on the other hand more than $r_2+1$ rows come from the bottom half, we can achieve a nontrivial representation of $\bm{0}$ only using these rows.
Since one of those two cases must happen, we have the required upper bound.
\end{proof}

The next result shows how positivity of $A$ can (under some conditions) be passed on to a compounded matrix.

\begin{lemma}\label{lemma:compMatPos}
Let $m>r$ be positive integers, $A\in\Z^{r\times m}$ be positive, and $Q\subset[m]$.
If there exists an integer vector $\bm{y}\in S(A^{\widebar{Q}},\bm{0})$ with all coordinates nonzero, then $A\overset{\id_Q}{\times}A$ is positive.
\end{lemma}

\begin{proof}
Since $A$ was positive, there exists a solution $\bm{x}=(x_1,\dots,x_m)\in S(A,\bm{0})\cap\N^m$.
But then if we define $\bm{y}=(\bm{x}^{\widebar{Q}}, \bm{x}^Q, \bm{x}^{\widebar{Q}})$, we clearly have $\bm{y}\in S(A\overset{\id_Q}{\times}A,\bm{0})\cap\N^m$.
It remains to be proved that for any two distinct indices $i,j\in[2m-|Q|]$, there exists a solution $(x_1,\dots,x_{2m-|Q|})=\bm{x}$ in $S(A\overset{\id_Q}{\times}A,\bm{0})$ such that $x_i\neq x_j$.
Let $\rho\colon[2m-|Q|]\to[m]$ denote the surjective function that maps the column indices of $A\overset{\id_Q}{\times}A$ to the corresponding ones of $A$.
If $\rho(i)\neq\rho(j)$, that is, $i$ and $j$ refer to distinct variables of $A$, this again just follows directly from the fact that $A$ was positive itself.
So we need to check the case $\rho(i)=\rho(j)\in\widebar{Q}$.
By our assumption on $A^{\widebar{Q}}$, there exists a vector $\bm{y}\in S(A^{\widebar{Q}},\bm{0})$ with all its coordinates nonzero.
Clearly, the same will hold for $2\bm{y}$, and hence \[\bm{x}=(\bm{y}, \bm{0}, 2\bm{y})\in\Z^{2m-|Q|}\] will satisfy $(A\overset{\id_Q}{\times}A)\bm{x}=\bm{0}$ and the $i$-th and $j$-th entry of $\bm{x}$ are different.
\end{proof}

Note that if $A\in\Z^{r\times m}$ is abundant, then any $Q\subset[m]$ of size $|Q|=1$ will satisfy the requirements of Lemma~\ref{lemma:compMatPos}.
The next result shows that even in the case that $Q$ itself does not satisfy them, we can find a superset that does, which will help us to get universal lower bounds that are sufficient for our applications.

\begin{lemma}\label{lemma:compound_matrix_LB}
Let $m>r$ be positive integers, $n\in\N$, $A\in\Z^{r\times m}$ positive, and $\bm{b}\in\Z^r$ such that $S(A,\bm{b})\neq\emptyset$.
Then for all $Q\subset [m]$ there exists $Q'\supset Q$ such that \[\big|S_0(A\overset{\id_{Q'}}{\times}A, \bm{b}^2)\cap[n]^{2m-|Q'|}\big| = \Omega\big(n^{2m-\rk(A)-\rk(A^{\widebar{Q}})-|Q|}\big).\]
\end{lemma}

\begin{proof}
If $Q$ satisfies the assumptions of Lemma~\ref{lemma:compMatPos}, we see that $A\overset{\id_Q}{\times}A$ is positive and hence we can apply the lower bound of Lemma~\ref{lemma:proper_sols_count} directly with $Q'=Q$, noting that $A\overset{\id_Q}{\times}A$ has $2m-|Q|$ columns and rank $\rk(A)+\rk(A^{\widebar{Q}})$ by Lemma~\ref{lemma:compMatRankUB}.
Otherwise, denote by $\bm{c}_1,\dots,\bm{c}_m\in\Z^r$ the columns of $A$, and let $Q_1\subset \widebar{Q}$ be the index set of the columns that will always have coefficient zero in a linear combination of the zero vector.
We claim that these columns are linearly independent and their span does not contain any column $\bm{c}_i$ with $i\in \widebar{Q}\setminus Q_1$.
Indeed, linear independence holds because any non-trivial linear combination $\sum_{i\in Q_1}\lambda_i \bm{c}_i=\bm{0}$ could be extended to a linear combination $\sum_{i\in\widebar{Q}}\lambda_i \bm{c}_i=\bm{0}$ such that $\lambda_i \neq 0$ for at least one $i\in Q_1$, a contradiction to the definition of $Q_1$.
Similarly, if $\sum_{i\in Q_1}\lambda_i \bm{c}_i = \bm{c}_j$ for some $j\in \widebar{Q}\setminus Q_1$, we obviously see that $\sum_{i\in Q_1}\lambda_i \bm{c}_i - \bm{c}_j=\bm{0}$ is a linear combination with at least one $\lambda_i\neq 0$, again a contradiction.

We thus see that defining $Q'=Q\cup Q_1$, the matrix $A^{\widebar{Q}'}$ has rank $\rk(A^{\widebar{Q}'})=\rk(A^{\widebar{Q}})-|Q_1|$ and its number of columns is $m-|Q'|=m-|Q|-|Q_1|$, and hence \[m-|Q'|-\rk(A^{\widebar{Q}'}) = m-|Q|-\rk(A^{\widebar{Q}}).\]
By construction, $A^{\widebar{Q}'}$ satisfies the assumption of Lemma~\ref{lemma:compMatPos}, and hence for $n\in\N$, applying Lemma~\ref{lemma:proper_sols_count} implies \[\big|S_0(A\overset{\id_{Q'}}{\times}A,\bm{b}^2)\cap[n]^{2m-|Q'|}\big| = \Omega\big(n^{2m-\rk(A)-\rk(A^{\widebar{Q}})-|Q|}\big),\] which is what we wanted to show.
\end{proof}

Lemmas~\ref{lemma:compMatRankUB} and~\ref{lemma:compMatPos} actually apply to a more general iterated construction.
Using the assumptions made in Lemma~\ref{lemma:compMatRankLB}, write $Q=\{q_1<\dots<q_{|Q|}\}$, and define for any integer $j\geq 1$ the bijection $M_j\colon [j(m-|Q|)+|Q|]\setminus[j(m-|Q|)]\to [|Q|]$ by $M_j(i)=q_{i-j(m-|Q|)}$.
Then all the results mentioned in the aforementioned lemmas also apply in a natural way to the matrix
\begin{equation}\label{eq:milkyWayCompMat}
A\overset{\id_Q}{\times}A\overset{M_1}{\times}\cdots\overset{M_t}{\times}A = \begin{pmatrix}A^{\widebar{Q}} & & &A^Q\\ &\ddots & &\vdots \\ & &A^{\widebar{Q}} &A^Q \\ & & &A^Q &A^{\widebar{Q}}\end{pmatrix}
\end{equation}
for any $t\geq 1$.
The idea being that instead of just having a pair of proper solutions to $A\bm{x}=\bm{b}$, we now have a collection of $t+1$ of them that all mutually intersect exactly in the variables indexed by $Q$.
Specifically, we get the following result.

\begin{lemma}\label{lemma:milkyWayCompMat}
Let $m>r$ be positive integers, $A\in\Z^{r\times m}$ be positive and abundant and $Q=\{i\}\subset[m]$ of size $1$, and for any positive integer $j$ let $M_j$ denote the function $jm-j+1\mapsto i$.
Then for any positive integer $t$ it holds that the matrix $A\overset{\id_Q}{\times}A\overset{M_1}{\times}\cdots\overset{M_t}{\times}A$ is positive and of rank \[\rk\bigl(A\overset{\id_Q}{\times}A\overset{M_1}{\times}\cdots\overset{M_t}{\times}A\bigr) = \rk(A)+(t+1)\rk(A^{\widebar{Q}}).\]
\end{lemma}

\begin{proof}
The proof is identical to that of Lemmas~\ref{lemma:compMatRankUB} and~\ref{lemma:compMatPos}, noting that since $A$ is abundant, for any $Q\subset [m]$ of size $1$ it will hold that $S(A^{\widebar{Q}},\bm{0})$ contains a solution with all entries non-zero.
\end{proof}

\section{The distribution of type-\texorpdfstring{$\frakP$}{P} solutions: the main meta theorem}\label{sec:mainProof}

We now state our main result (which we call \emph{meta theorem}) that will imply Theorems~\ref{thm:distProper} and~\ref{thm:distHomStrictBal}.

\begin{thm}\label{thm:distribution}
Let $m>r$ be positive integers, $A\in\Z^{r\times m}$ a positive and abundant integer matrix, $\bm{b}\in\Z^r$ such that $S(A,\bm{b})\neq\emptyset$, and $\frakP\subset \frakP(A)$ satisfying $(\{1\},\dots,\{m\})\in\frakP$ and \[\{\frakp\in\frakP : S_0(A_\frakp,\bm{b})\neq\emptyset\}\subset \frakP_0(A).\]
Furthermore, let $n$ be an integer, $0\leq p:=p(n)\leq 1$ and $X_n$ the random variable $|S_\frakP(A,\bm{b})\cap[n]_p^m|$ that counts the number of type--$\frakP$ solutions $\bm{x}\in[n]_p^m$ to $A\bm{x}=\bm{b}$.
Then \[\tilde{X}_n \xrightarrow{d} \calN(0,1)\] if $n(1-p)\to\infty$ and $np^{c(A_\frakp)}\to\infty$ for all $\frakp\in\frakP$ with $S_\frakp(A,\bm{b})\neq\emptyset$.
\end{thm}

Let us first see that this indeed implies Theorems~\ref{thm:distProper} and~\ref{thm:distHomStrictBal}.

\begin{proof}[Proof of Theorem~\ref{thm:distProper}]
We apply Theorem~\ref{thm:distribution} with $\frakP=\{(\{1\},\dots,\{m\})\}$, noting that $S_0(A,\bm{b})\neq\emptyset$ follows from the assumptions and Lemma~\ref{lemma:proper_sols_count}.
\end{proof}

\begin{proof}[Proof of Theorem~\ref{thm:distHomStrictBal}]
We see that for any $\frakp\in\frakP(A)$, whenever $S_0(A_\frakp,\bm{0})\neq \emptyset$, the matrix $A_\frakp$ is positive by definition, that is $\frakp\in\frakP_0(A)$.
Since $A$ is strictly balanced, it holds that \[c(A) = \frac{m}{m-\rk(A)} > \max_{\emptyset \subsetneqq Q\subsetneqq [m]} \frac{|Q|}{|Q|-r_Q(A)} > c(A_\frakp)\] for any $\frakp\in\frakP(A)$ with $|\frakp|<m$ and hence $np^{c(A_\frakp)}\to\infty$.
\end{proof}

The proof of Theorem~\ref{thm:distribution} involves the analysis of several sub-cases, which we will split into different subsections.
In general, it will follow the ideas that were used by Ruci\'{n}ski in~\cite{Rucinski1988} where he proved a similar result to Theorem~\ref{thm:distribution} in order to determine the distribution of the number of occurrences of a fixed graph $G$ as a subgraph in a binomial random graph.

\subsection{The proof of Theorem~\ref{thm:distribution}}\label{sec:mainProofIf}
Before going into further case analysis, let us first discuss the common jumping off point and strategy.
Let $\mu_k=\Exp((X_n-\Exp(X_n))^k)$ denote the $k$-th central moment of $X_n$ associated to the system of linear equations $A \textbf{x}=\bm{b}$.
Our final goal will always be to show that, independently of the system studied,
\begin{equation}\label{eq:moment_goal}
\mu_k = \begin{cases}(1+o(1))\frac{k!}{(k/2)!}2^{-k/2}\mu_2^{k/2} &\text{ if $k$ is even,}\\ o(\mu_2^{k/2})&\text{ if $k$ is odd.}\end{cases}
\end{equation}
These estimates would show that the moments $\Exp(\tilde{X}_n^k)$ converge to the moments of a normal distribution, which is uniquely determined by its moments.

Given a solution $\bm{x}=(x_1,\dots,x_m)$ in $S(A,\bm{b})$, denote by $\Ind_{\bm{x}}$ the indicator random variable for the event $\{x_1,\dots,x_m\}\subset[n]_p$.
Abusing notation somewhat, if $\chi=(\bm{x}_1,\dots,\bm{x}_k)\in S(A,\bm{b})^k$ is a $k$-tuple of solutions, we will write $\{\chi\}=\bigcup_{i=1}^k\{\bm{x_i}\}$.
As a visual shorthand, bold latin letters will indicate solutions, while greek letters will denote tuples of solutions.
Note first that by definition we have
\begin{equation}\label{eq:kth_moment}
\mu_k = \sum_{\chi\in\mathfrak{S}_k} \Exp\left(\prod_{\bm{x}\in\chi}\left(\Ind_{\bm{x}}-p^{|\{\bm{x}\}|}\right)\right),
\end{equation}
where $\mathfrak{S}_k$ is the set containing all $k$-tuples $\chi={(\bm{x}_1,\dots,\bm{x}_k)\in (S_\frakP(A,\bm{b})\cap[n]^m)^k}$ such that for every $i\in [k]$ there exists a $j\in[k]\setminus\{i\}$ such that $\{\bm{x}_i\}\cap \{\bm{x}_j\}\neq \emptyset$.
Indeed, for any $k$-tuple $\chi$ that contains some $\bm{x}\in S_\frakP(A,\bm{b})\cap[n]^m$ that is disjoint from the remaining $k-1$ solutions, we have \[\Exp\left(\prod_{\bm{y}\in\chi}\left(\Ind_{\bm{y}}-p^{|\{\bm{y}\}|}\right)\right) = \Exp\left(\Ind_{\bm{x}}-p^{|\{\bm{x}\}|}\right)\Exp\Bigg(\prod_{\substack{\bm{y}\in\chi\\ \bm{y}\neq\bm{x}}}\left(\Ind_{\bm{y}}-p^{|\{\bm{y}\}|}\right)\Bigg)=0.\]
Equation~\eqref{eq:kth_moment} behaves slightly different depending on the behavior of $p$, so suppose $p\to a$ for some constant $a\in[0,1]$. We split the analysis in three cases, depending on wether this limit belongs to $(0,1)$, is equal to 1 and is equal to 0.

\subsubsection{Case 1: \texorpdfstring{$\mathbf{0<a<1}$}{0<a<1}.}\label{proof: case0<a<1}

In this case we see that~\eqref{eq:kth_moment} implies $\mu_k = \sum_{\mathfrak{S}_k} \Theta(1)$, so we need to analyze the cardinality of $\mathfrak{S}_k$.
We are going to prove~\eqref{eq:moment_goal} by induction on $k$.
The base cases $k=2$ and $k=1$ are clearly always true.
Our induction hypothesis tells us then that for any $\ell<k$ it holds that $\mu_\ell = O(\mu_2^{\ell/2})$ and hence by our previous observation
\[|\frakS_\ell|=O(|\frakS_2|^{\ell/2}).\]
Let us now study the statement for parameter $k$. The analysis will be different depending on wether $k$ is odd or even.  
Let $\frakS_k'$ denote the subset of $\frakS_k$ containing all $k$-tuples of solutions $(\bm{x}_1,\dots,\bm{x}_k)$ with $\bm{x}_i\in(S_\frakP(A,\bm{b})\cap[n]^m)$ such that for every $i$ the choice of $j\neq i$ with $\{\bm{x}_i\}\cap \{\bm{x}_j\}\neq \emptyset$ is unique, that is, the solutions can be grouped into $k/2$ pairwise disjoint pairs.
Note that clearly, $\frakS_k'=\emptyset$ for every odd $k$.
If $\widebar{\frakS}_k'=\frakS_k\setminus \frakS_k'$, we will show that
\begin{equation*}\label{eq:if_case1_nonpaired_negligible}
|\widebar{\frakS}_k'| = o(|\frakS_2|^{k/2}),
\end{equation*}
which would imply \[\sum_{\chi\in\widebar{\frakS}_k'}\Exp\left(\prod_{\bm{x}\in\chi}\left(\Ind_{\bm{x}}-p^{|\{\bm{x}\}|}\right)\right) = o(\mu_2^{k/2}).\]
To see that this is true, suppose $\chi:=(\bm{x}_1,\dots,\bm{x}_k)\in\widebar{\frakS}_k'$.
Then there must exist an index $i=i(\chi) \in [k]$ such that \[\chi^{[k]\setminus\{i\}}=(\bm{x}_1,\dots,\bm{x}_{i-1},\bm{x}_{i+1},\dots,\bm{x}_k)\] is contained in $\frakS_{k-1}$.
Taking the minimum choice of $i$ for each $\chi$, we have thus defined a map $\pi\colon \widebar{\frakS}_k'\to\frakS_{k-1}$, and hence \[|\frakS_k'| \leq \max_{\upsilon\in\frakS_{k-1}}|\pi^{-1}(\upsilon)|\cdot|\frakS_{k-1}|.\]
Since by induction $|\frakS_{k-1}|=O(|\frakS_2|^{(k-1)/2})$, it suffices to show that
 $$\max_{\upsilon\in\frakS_{k-1}}|\pi^{-1}(\upsilon)|=o(|\frakS_2|^{1/2}).$$
We will first determine a lower bound on $|\frakS_2|$.
It is clear that $|\frakS_2|\geq |S_0(A\overset{\id_Q}{\times}A,\bm{b}^2)\cap[n]^{2m-|Q|}|$ for any $Q\subset[m]$.
Since $A$ is abundant, any $Q$ containing exactly one index will satisfy the conditions of Lemma~\ref{lemma:compMatPos}, so fix an arbitrary one.
It follows by this and Lemma~\ref{lemma:proper_sols_count} applied to $A\overset{\id_Q}{\times}A$ that
\begin{equation}\label{eq:if_case1_mu2_LB}
|\frakS_2| \geq |S_0(A\overset{\id_Q}{\times}A,\bm{b}^2)\cap[n]^{2m-1}|=\Omega(n^{2m-2\rk(A)-1}).
\end{equation}
We now turn to giving an upper bound on $\max_{\upsilon\in\frakS_{k-1}}|\pi^{-1}(\upsilon)|$, so fix an $\upsilon\in\frakS_{k-1}$.
By definition of $\frakS_k$, any solution that could extend $\upsilon$ to a $k$-tuple in $\frakS_k$ must intersect $\{\upsilon\}$, and hence by Lemma~\ref{lemma:sols_with_fixed_elts} we see that \[|\pi^{-1}(\upsilon)| = O(n^{m-\rk(A)-1}) = o\left(|\frakS_2|^{1/2}\right)\] by the previously obtained lower bound~\eqref{eq:if_case1_mu2_LB}.
Since $\frakS_k'=\emptyset$ for odd $k$, this also proves the odd case of~\eqref{eq:moment_goal}.

We now turn to the case of even $k$, noting that~\eqref{eq:if_case1_nonpaired_negligible} tells us that essentially, the tuples that are summed over in $\mu_k$ are $k/2$ pairwise disjoint pairs of those summed over in $\mu_2$.
We begin by showing that for all but a negligible amount, pairs $(\bm{x},\bm{y})\in\frakS_2$ will satisfy $\bm{x},\bm{y}\in S_0(A,\bm{b})$ and $|\{\bm{x}\}\cap \{\bm{y}\}|=1$.
Let us first see that we can restrict ourselves to pairs of proper solutions.
This follows from a similar argument as was used in the proof of Lemma~\ref{lemma:sols_with_fixed_elts}.
Namely, if we fix an arbitrary $\bm{x}\in S_\frakP(A,\bm{b})$, then by~\eqref{eq:sols_with_fixed_elts_positive}
 \[\big| \{y\in S_0(A_\frakp, \bm{b})\cap [n]^{|\frakp|} : \frakp\in\frakP,\, |\frakp|<m\} = O(n^{m-\rk(A)-2}).\]
Since there are at most $n^{m-\rk(A)}$ choices for $\bm{x}$, this is negligible when compared to $|\frakS_2|$.
Applying the second part of Lemma~\ref{lemma:sols_with_fixed_elts} directly, the number of pairs of proper solutions that intersect in at least two elements is negligible as well.
Note that for any pair $(\bm{x},\bm{y})$ satisfying the structure described above, we have $\Exp((\Ind_{\bm{x}}-p^{\{\bm{x}\}})(\Ind_{\bm{y}}-p^{\{\bm{y}\}}))\sim (a^{2m-1}(1-a))$.

Finally, note that for any pair $\chi\in\frakS_2$, the number of pairs $\upsilon\in\frakS_2$ that share elements with $\chi$ is negligible:
Clearly, any such $\upsilon$ will consist of an $\bm{x}\in S_0(A,\bm{b})$ satisfying $\{\bm{x}\}\cap \{\chi\}\neq\emptyset$ and a $\bm{y}\in S_0(A,\bm{b})$ satisfying $\{\bm{y}\}\cap\{\bm{x}\}\neq\emptyset$.
Applying Lemma~\ref{lemma:sols_with_fixed_elts} again, the number of such pairs is $O(n^{2m-2\rk(A)-2})=o(|\frakS_2|)$.
We conclude
\[|\frakS_{k}| \sim \binom{|\frakS_2|/2}{k/2}k!\]
and hence, as $p$ tends to $a\in (0,1)$,
\[\mu_k \sim \binom{|\frakS_2|/2}{k/2}k!(a^{2m-1}(1-a))^{k/2} \sim \frac{k!}{(k/2)!}2^{-k/2}\mu_2^{k/2},\]
which is what we wanted to prove.
\hfill\qedsymbol 

\subsubsection{Case 2: \texorpdfstring{$\mathbf{a=1}$}{a=1}.}\label{proof: casea=1}

Recall that $p$ is chosen such that $n(1-p)$ tends to infinity. This property will be specially relevant in this case.
We will instead look at the complements, so if $\bm{x}$ is a solution to $A\bm{x}=\bm{b}$, define $q=1-p$ and $\bar{\Ind}_{\bm{x}} = 1-\Ind_{\bm{x}}$.
We see that $\nu_{\bm{x}} \vcentcolon = \Exp(\bar{\Ind}_{\bm{x}})=1-p^{|\{\bm{x}\}|}\sim |\{\bm{x}\}|q$ when $n$ tends to infinity.
Let $\chi=(\bm{x}_1,\dots,\bm{x}_k)\in\frakS_k$ be a $k$-tuple of solutions in $S_\frakP(A,\bm{b})\cap[n]^m$.
Using these definitions we see that~\eqref{eq:kth_moment} can be written as
\begin{equation}\label{eq:if_case2_moment_terms}
\mu_k=\sum_{\chi\in\frakS_k}\Exp\left(\prod_{\bm{x}\in\chi}(\Ind_{\bm{x}}-p^{|\{\bm{x}\}|})\right) = (-1)^k\sum_{\chi\in\frakS_k} \Exp\left(\prod_{\bm{x}\in\chi} (\bar{\Ind}_{\bm{x}}-\nu_{\bm{x}})\right).
\end{equation}
We associate then a set system $\calH=\calH(\chi)$ to the $k$-tuple $\chi$ as follows: 
the vertex set $V(\calH)$ is just $\{\chi\}$, while the edges are $E(\calH)=\{\{\bm{x}\} : \bm{x}\in\chi\}$.
Suppose the vertex cover number $\tau(\calH)$ of the hypergraph is $s$, that is, there exist $s$ elements $a_1,\dots,a_s\in\{\chi\}$ such that removing them destroys all solutions in $\chi$, and no collection of less than $s$ elements in $[n]$ achieves this.
Then a straightforward computation shows that
\[\Exp\left(\prod_{\bm{x}\in\chi}\bar{\Ind}_{\bm{x}}\right) = \Theta(q^s).\]
Furthermore, it is clear that for any $t\in[k-1]$, every $(k-t)$-sub-collection of $\chi$ will require removal of at least $s-t$ elements to destroy all solutions in it, that is, for any $I\subset [k]$ with $|I|=k-t$ it holds that \[\Exp\left(\prod_{i\in I}\bar{\Ind}_{\bm{x}_i}\prod_{j\notin I}\nu_{\bm{x}_j}\right) = \begin{cases}\Theta(q^{s}), &\text{if $\chi^{I}\in\frakS_{k-t}$,}\\ 0, &\text{otherwise.} \end{cases}\]
Note here, we interpret $\chi$ as a $k$-vector in which every coordinate is itself an $m$-vector, so $\chi^I$ denotes the vector obtained by only keeping the solutions indexed by $I$.
Putting this together, we see that~\eqref{eq:if_case2_moment_terms} becomes
\begin{equation}\label{eq:if_case2_KthMomentAsymp}
\mu_k = \sum_{s=1}^{k} \sum_{\chi\in\frakS_{k,s}} \Theta(q^s),
\end{equation}
where $\frakS_{k,s}$ is the subset of $k$-tuples $\chi\in\frakS_k$ such that $\tau(\calH(\chi))=s$.
We will now show that for any fixed $s$, only $\chi$ of a certain structure contribute significantly to~\eqref{eq:if_case2_KthMomentAsymp}.
For this, define by $\frakS_{k,s}'$ the set of \emph{$s$-milky ways}, which are $\chi\in\frakS_{k,s}$ such that $\calH(\chi)$ is the union of $s$ disjoint components, with all edges in a component intersecting in a unique vertex.
Note that this set is only non-empty for $s\leq\floor{k/2}$.
Furthermore, note that each component corresponds to a matrix as described in~\eqref{eq:milkyWayCompMat}, call it $B_i$, hence by considering the compounded matrix $B_1 \overset{\bm{.}}{\times} \cdots \overset{\bm{.}}{\times}B_s$ we see that Lemma~\ref{lemma:milkyWayCompMat} together with Lemma~\ref{lemma:proper_sols_count} implies that for any $k$ and $s\leq \floor{k/2}$
\begin{equation}\label{eq:if_case2_milkywayLB}
|\frakS_{k,s}'| = \Omega(n^{k(m-\rk(A)-1)+s}).
\end{equation}
The combination of these two lemmas also gives a matching upper bound:
For any fixed way of dividing up the $k$ solutions into $s$ components, there are at most $O(n^{k(m-\rk(A)-1)+s})$ corresponding $k$-tuples.
Since $k$ and $s$ are fixed, the number of these partitions is $O(1)$, and hence
\begin{equation}\label{eq:if_case2_milkywayUB}
|\frakS_{k,s}'| = O(n^{k(m-\rk(A)-1)+s}).
\end{equation}
We will now show by induction on $k$ that for any $k$ and $s$, \[|\widebar{\frakS}_{k,s}'|\coloneqq|\frakS_{k,s}\setminus \frakS_{k,s}'|=O(n^{k(m-\rk(A)-1)+s-1}).\]
This holds trivially for $k=1$, and for $k=2$, it follows immediately from Lemma~\ref{lemma:sols_with_fixed_elts}, since $A$ is abundant and $s=1$ being the only vertex cover number leading to a non-empty set.
So suppose $k\geq 3$ and note that the bound~\eqref{eq:if_case2_milkywayUB} together with the induction hypothesis in particular implies $|\frakS_{\ell,s}|=O(n^{\ell(m-\rk(A)-1)+s})$ for any $\ell<k$.

We will split up $\widebar{\frakS}'_{k,s}$ even further, into disjoint sets $\widebar{\frakS}_{k,s}''$, $\widebar{\frakS}_{k,s}'''$, and $\widebar{\frakS}_{k,s}''''$, which are defined as follows:

\begin{itemize}
\item[a)] The set $\widebar{\frakS}_{k,s}''$ will contain all $\chi=(\bm{x}_1,\dots,\bm{x}_k)\in\widebar{\frakS}_{k,s}'$ such that there are solutions $\bm{x}_i, \bm{x}_j\in\chi$ satisfying $\{\chi^{[k]\setminus\{i,j\}}\}\cap\{\chi^{\{i,j\}}\}=\emptyset$, and $|\{\bm{x}_i\}\cap\{\bm{x}_j\}|\geq 2$.
\item[b)] The set $\widebar{\frakS}'''_{k,s}$ contains all $\chi=(\bm{x}_1,\dots,\bm{x}_k)\in\widebar{\frakS}_{k,s}'\setminus\widebar{\frakS}''_{k,s}$, such that there exists a solution $\bm{x}_i\in\chi$ satisfying $|\{\bm{x}\}\cap\{\chi^{[k]\setminus\{i\}}\}|\geq 2$ and $\chi^{[k]\setminus\{i\}}\in \frakS_{k-1,s}\cup\frakS_{k-1,s-1}$.
\item[c)] Finally, $\widebar{\frakS}''''_{k,s}=\widebar{\frakS}'_{k,s}\setminus (\widebar{\frakS}''_{k,s}\cup\widebar{\frakS}'''_{k,s})$ are the remaining non-milky ways.
\end{itemize}

We proceed to analyze the cardinality of each set.

\begin{itemize}

\item[a)] We start with case $\widebar{\frakS}_{k,s}''$.
For every $\chi\in\widebar{\frakS}''_{k,s}$, there are indices $i$ and $j$ such that $\chi^{[k]\setminus\{i,j\}}\in \frakS_{k-2,s-1}$, and taking the lexicographic smallest pair we have defined a map $\pi\colon \widebar{\frakS}''_{k,s}\to\frakS_{k-2,s-1}$, which implies \[|\widebar{\frakS}''_{k,s}| \leq |\frakS_{k-2,s-1}|\max_{\upsilon\in\frakS_{k-2,s-1}}|\pi^{-1}(\upsilon)|=O\left(n^{(k-2)(m-\rk(A)-1)+s-1}\max_{\upsilon\in\frakS_{k-1,s}}|\pi^{-1}(\upsilon)|\right).\]
Lemmas~\ref{lemma:proper_sols_count} and~\ref{lemma:sols_with_fixed_elts} imply that for any $\upsilon\in\frakS_{k-2,s-1}$ we have $|\pi^{-1}(\upsilon)|=O(n^{2m-2\rk(A)-2})$, and hence
\begin{equation}\label{eq:if_case2_primeprimeUB}
|\widebar{\frakS}''_{k,s}| = O(n^{k(m-\rk(A)-1)+s-1}).
\end{equation}

\item[b)] We continue with the analysis of $\widebar{\frakS}'''_{k,s}$.
In this case, taking the smallest possible index $i$, this again defines a map $\pi\colon \widebar{\frakS}'''_{k,s}\to \frakS_{k-1,s}\cup\frakS_{k-1,s-1}$ and we see by induction hypothesis that
\begin{align*}
|\widebar{\frakS}'''_{k,s}|&\leq |\frakS_{k-1,s}\cup\frakS_{k-1,s-1}|\max_{\upsilon\in\frakS_{k-1,s}\cup\frakS_{k-1,s-1}}|\pi^{-1}(\upsilon)|\\
&= O\left(n^{(k-1)(m-\rk(A)-1)+s}\max_{\upsilon\in\frakS_{k-1,s}\cup\frakS_{k-1,s-1}}|\pi^{-1}(\upsilon)|\right).
\end{align*}
Again, for any $\upsilon\in \frakS_{k-1,s}\cup\frakS_{k-1,s-1}$, we see that $|\pi^{-1}(\upsilon)|$ is at most the number of solutions $\bm{y}\in S_\frakP(A,\bm{b})\cap[n]^m$ that contain at least $2$ elements from $\{\upsilon\}$, which is $O(n^{m-\rk(A)-2})$ by Lemma~\ref{lemma:sols_with_fixed_elts}, and hence
\begin{equation}\label{eq:if_case2_primeprimeprimeUB}
|\widebar{\frakS}'''_{k,s}| = O(n^{k(m-\rk(A)-1)+s-1}).
\end{equation}

\item[c)] Finally, it remains to look at $\widebar{\frakS}''''_{k,s}$.
If $\chi=(\bm{x}_1,\dots,\bm{x}_k)$ in $\widebar{\frakS}''''_{k,s}$, we claim that there must exist an index $i$ such that $\chi^{[k]\setminus\{i\}}\in\frakS_{k-1,s-1}$.
In the sequel, we will only consider the components of $\calH(\chi)$ that do not intersect in a unique vertex, of which there is at least one since $\chi\notin\frakS'_{k,s}$.
First note that it is clear that since $\chi\notin\widebar{\frakS}'_{k,s}\cup\widebar{\frakS}''_{k,s}$, there must exist an index $i$ such that $\chi^{[k]\setminus\{i\}}\in\frakS_{k-1,s-1}\cup\frakS_{k-1,s}$ or in other words, the remaining solutions are still intersecting.
Since $\chi\notin\widebar{\frakS}'''_{k,s}$, for all of these indices, it must hold that $|\{\bm{x}_i\}\cap\{\chi^{[k]\setminus\{i\}}\}|=1$.
Finally, for at least one index $i$ of the previously considered, it must actually hold that there is a unique $j_i$ such that
\begin{equation}\label{eq:if_case2_nonMW3}
\{\bm{x}_i\}\cap\{\chi^{[k]\setminus\{i\}}\} = \{\bm{x}_i\}\cap\{\bm{x}_{j_i}\}.
\end{equation}
Indeed, since the components we consider are not sunflowers, there must exist an $\bm{x}_\ell$ that intersects the rest of its component in at least two points, and so the negation of~\eqref{eq:if_case2_nonMW3} would imply $\chi\in\widebar{\frakS}'''_{k,s}$, since $\bm{x}_\ell$ would be a valid choice.
We see that for any valid $i$ satisfying~\eqref{eq:if_case2_nonMW3} it holds that $\chi^{[k]\setminus\{i\}}\in\frakS_{k-1,s-1}$.
Having established this, we repeat the arguments already used, taking the minimal valid $i$ and defining the appropriate function $\pi\colon \widebar{\frakS}''''_{k,s}\to\frakS_{k-1,s-1}$ and conclude
\begin{align*}
|\widebar{\frakS}''''_{k,s}| &\leq |\frakS_{k-1,s-1}|\max_{\upsilon\in\frakS_{k-1,s-1}}|\pi^{-1}(\upsilon)|\\
&=O(n^{(k-1)(m-\rk(A)-1)+s-1})\max_{\upsilon\in\frakS_{k-1,s-1}}|\pi^{-1}(\upsilon)|\\
&= O(n^{k(m-\rk(A)-1)+s-1}).
\end{align*}
\end{itemize}


Together with~\eqref{eq:if_case2_milkywayLB}, we see that~\eqref{eq:if_case2_primeprimeUB},~\eqref{eq:if_case2_primeprimeprimeUB} and~\eqref{eq:if_case2_nonMW3} imply that only $s$-milky ways contribute meaningfully when $s\leq \floor{k/2}$.
Sadly, this bound is not quite strong enough when $s> \floor{k/2}$, since $s$-milky ways do not exist here.
For this, we will prove by induction on $k$ that for any $k$ and $s \geq \floor{k/2}$ it holds that
\begin{equation}\label{eq:if_case2_bigcover}
|\frakS_{k,s}| = O(n^{k(m-\rk(A)-1)+\floor{k/2}}).
\end{equation}
Again, the cases $k=1$ and $k=2$ follow from previous observations.
Moreover, the statement is trivially true for any $k$ and $s=\floor{k/2}$, so suppose $k\geq 3$ and $s > \floor{k/2}$.
If $\chi=(\bm{x}_1,\dots,\bm{x_k})\in\frakS_{k,s}$, then there must exist a least index $i$ such that $\chi^{[k]\setminus\{i\}}\in\frakS_{k-1,s}\cup\frakS_{k-1,s-1}$, which allows us again to define a projection map $\pi\colon\frakS_{k,s}\to  \frakS_{k-1,s}\cup\frakS_{k-1,s-1}$.
Since $s\geq \floor{k/2}+1$ we have $s-1\geq \floor{(k-1)/2}$, and hence by induction hypothesis we have
\begin{align*}
|\frakS_{k,s}| &\leq |\frakS_{k-1,s}\cup\frakS_{k-1,s-1}|\max_{\upsilon\in\frakS_{k-1,s}\cup\frakS_{k-1,s-1}}|\pi^{-1}(\upsilon)|\\
&= O(n^{(k-1)(m-\rk(A)-1)+\floor{(k-1)/2}})|\max_{\upsilon\in\frakS_{k-1,s}\cup\frakS_{k-1,s-1}}|\pi^{-1}(\upsilon)|\\
&= O(n^{k(m-\rk(A)-1)+\floor{(k-1)/2}}),
\end{align*}
which in particular implies~\eqref{eq:if_case2_bigcover}.
Since $q\to 0$, we thus have \[|\frakS_{k,s}|q^s=o\bigl(|\frakS_{k,\floor{k/2}}|q^{\floor{k/2}}\bigr),\] and so the terms with $s>\floor{k/2}$ in~\eqref{eq:if_case2_KthMomentAsymp} can be discarded.
Now, if $\chi\in\frakS'_{k,s}$ is a milky way, we see that \[\Exp\left(\prod_{\bm{x}\in\chi}\bar{\Ind}_{\bm{x}}\right) \sim q^s,\] while for any $I\in[k]$ with $|I|\leq k-1$ we have \[\Exp\left(\prod_{i\in I}\bar{\Ind}_{\bm{x}_i}\prod_{j\notin I}\nu_{\bm{x}_j}\right) = O(q^{s+1}) = o(q^s),\] since the removal of any single solution does not impact the cover number, while any subsequent removal can reduce it by at most one.
Furthermore, by our previous observations we have already seen that for any $k$ and $s\leq\floor{k/2}$ \[|\frakS'_{k,s}|q^s=\Theta(n^{k(m-\rk(A)-1)}(nq)^s),\] so since $nq\to\infty$ by assumption only the $s=\floor{k/2}$ term contributes significantly.
Hence~\eqref{eq:if_case2_KthMomentAsymp} can be rewritten as
\begin{equation}\label{eq:if_case2_momentsprecise}
\mu_k \sim \sum_{s=1}^{\floor{k/2}}|\frakS'_{k,s}|q^s \sim |\frakS'_{k,\floor{k/2}}|q^{\floor{k/2}}.
\end{equation}
This proves~\eqref{eq:moment_goal} for odd $k$, since \[\mu_k = O(n^{k(m-\rk(A)-1)}(nq)^{(k-1)/2})=o(n^{k(m-\rk(A)-1)}(nq)^{k/2}),\] while
\[\mu_2^{k/2}=\Omega(n^{k(m-\rk(A)-1)}(nq)^{k/2}).\]
For the even case, it is easy to see that one can repeat the argument from Case~1 to show that only those milky ways with all solutions being proper contribute meaningfully and that among those, the number of pairs in $\frakS'_{2,1}$ that intersect another pair is negligible, so just like in that case we conclude \[\mu_k \sim \binom{|\frakS'_{2,1}|/2}{k/2}k!q^{k/2} \sim \frac{k!}{(k/2)!}2^{-k/2}\mu_2^{k/2}.\]
\hfill\qedsymbol

\subsubsection{Case 3: \texorpdfstring{$\mathbf{a=0}$}{a=0}.}\label{proof: casea=0}

Finally, in this case the crucial property of $p$ that will be used is the fact that $np^{c(A_\frakp)}$ tends to infinity for every considered partition type $\frakp\in\frakP$ with non-empty solution set.

Let $\chi=(\bm{x}_1,\dots,\bm{x}_k)\in\frakS_k$.
Then for any $I\subsetneqq [k]$ we see that \[\Exp\Bigl(\prod_{i\in I}\Ind_{\bm{x}_i}\prod_{j\in[k]\setminus I}p^{|\{\bm{x}_j\}|}\Bigr) = p^{\left|\bigcup_{i\in I}\{\bm{x}_i\}\right|+\sum_{j\in[k]\setminus I}|\{\bm{x}_j\}|} = o\left(p^{\left|\bigcup_{i=1}^k \{\bm{x}_i\}\right|}\right),\] since every $\bm{x}_j$ with $j\notin I$ intersects at least one other solution.
Since we clearly have $\Exp(\Ind_{\bm{x}_1}\cdots\Ind_{\bm{x}_k})=p^{\left|\bigcup_{i=1}^k \{\bm{x}_i\}\right|}$, it thus follows that
\begin{equation}\label{eq:if_case3_moment}
\Exp\left((\mathbb{I}_{\bm{x}_1}-p^{|\{\bm{x}_1\}|})\cdots(\mathbb{I}_{\bm{x}_k}-p^{|\{\bm{x}_k\}|})\right) \sim p^{\left|\bigcup_{i=1}^k \{\bm{x}_i\}\right|}.
\end{equation}
Furthermore, since $np^{c(A_\frakp)}\to\infty$ for any $\frakp\in\frakP$, wee see that for any nonempty $Q\subset[|\frakp|]$, it holds that
\begin{equation*}\label{eq:if_case3_everysolutioninfinite}
n^{|Q|-r_Q(A_\frakp)}p^{|Q|}=\omega(1).
\end{equation*}
Let us prove~\eqref{eq:moment_goal} via induction on $k$, the cases $k=1$ and $k=2$ clearly being true.
Note that in particular, the induction hypothesis implies that for any $\ell<k$ it holds that $\mu_\ell=O(\mu_2^{\ell/2})$.
Let us decompose $\mu_k$ as
\[\mu_k \sim \sum_{\chi\in\frakS_k}\Exp\bigg(\prod_{\bm{x}\in\chi}\Ind_{\bm{x}}\bigg) = \sum_{\chi\in\frakS'_k}\Exp\bigg(\prod_{\bm{x}\in\chi}\Ind_{\bm{x}}\bigg) + B_k.\]
Recall that $\frakS'_k$ denoted the subset of $\frakS_k$ such that every $\bm{x}\in\chi$ had a unique partner $\bm{y}\in\chi$ and was disjoint from all other components.
Our first step will be to show that \[B_k = o(\mu_2^{k/2}),\] which would in particular imply~\eqref{eq:moment_goal} in the case of odd $k$, since here $\frakS'_k=\emptyset$.
To see this, note that $\chi=(\bm{x}_1,\dots,\bm{x}_k)\in\frakS_k\setminus\frakS'_k$ implies that there must exist an index $i\in[k]$ such that $\chi^{[k]\setminus\{i\}}\in\frakS_{k-1}$.
Let $i_\chi$ denote the least index $i\in[k]$ for which this is true, then we can define the map $\pi\colon \frakS_k\setminus \frakS'_k\to \frakS_{k-1}$ by $\pi(\chi)=\chi^{[k]\setminus\{i_\chi\}}$.
Furthermore, let $Q_\chi\subset[|\frakp(\bm{x}_{i_\chi})|]$ denote the index set of all the components of $\bm{x}_\chi$ that are contained in $\{\pi(\chi)\}$.
Then
\begin{equation}\label{eq:if_case3_oddcase}
\begin{split}
B_k &= \sum_{\chi\in\frakS_k\setminus\frakS'_k} \Exp\bigg(\prod_{\bm{x}\in\chi}\Ind_{\bm{x}}\bigg)\\
&= \sum_{\chi\in\frakS_k\setminus\frakS'_k} \Exp\bigg(\prod_{\bm{x}\in\pi(\chi)}\Ind_{\bm{x}}\bigg)p^{|\frakp(\bm{x}_{i_\chi})|-|Q_\chi|}\\
&\leq \sum_{\upsilon\in\frakS_{k-1}} \Exp\bigg(\prod_{\bm{y}\in\upsilon}\Ind_{\bm{y}}\bigg)\sum_{\frakp\in\frakP}\sum_{Q\subset[|\frakp|]}\sum_{\bm{z}\in\{\upsilon\}^{|Q|}}p^{|\frakp|-|Q|}\big|S_0(A_\frakp^{\widebar{Q}}, b-A_\frakp^Q\bm{z})\cap [n]^{|\frakp|-|Q|}\big|\\
&\leq \mu_{k-1} \sum_{\frakp\in\frakP}\sum_{Q\subset[|\frakp|]}n^{|\frakp|-|Q|-\rk(A_\frakp^{\widebar{Q}})}p^{|\frakp|-|Q|}\\
&= \mu_{k-1} \sum_{\frakp\in\frakP}\sum_{Q\subset[|\frakp|]}n^{|\frakp|-\rk(A)-(|Q|-r_Q(A_\frakp))}p^{|\frakp|-|Q|}
\end{split}
\end{equation}
Since $\mu_{k-1}=O(\mu_2^{(k-1)/2})$ it thus suffices to show that the remaining expression in~\eqref{eq:if_case3_oddcase} is $o({\mu_2}^{1/2})$.
But by Lemma~\ref{lemma:compound_matrix_LB} there exists a $Q'\supset Q$ such that, using $p\to 0$ we see that
\begin{align*}
\mu_2 &\geq p^{2|\frakp|-|Q'|}\big|S_0(A_\frakp\overset{\id_{Q'}}{\times}A_\frakp, (\bm{b},\bm{b}))\cap[n]^{2|\frakp|-|Q'|}\big|\\
&\geq p^{2|\frakp|-|Q|}\big|S_0(A_\frakp\overset{\id_{Q'}}{\times}A_\frakp, (\bm{b},\bm{b}))\cap[n]^{2|\frakp|-|Q'|}\big|\\
&= \Omega(n^{2|\frakp|-2\rk(A)-(|Q|-r_Q(A_\frakp))}p^{2|\frakp|-|Q|}),
\end{align*}
and hence \[{\mu_2}^{1/2} = \Omega(n^{|\frakp|-\rk(A)-(|Q|-r_Q(A_\frakp))/2}p^{|\frakp|-|Q|/2}) = \omega(n^{|\frakp|-\rk(A)-(|Q|-r_Q(A_\frakp))}p^{|\frakp|-|Q|})\] which follows from~\eqref{eq:if_case3_everysolutioninfinite}.
As stated before, this finishes the proof in the case of odd $k$, so suppose $k$ is even.
If $\frakp,\frakq\in\frakP$ and $M\colon P\to Q$ is a bijection between nonempty $P\subset[|\frakp|]$ and $Q\subset[|\frakq|]$, we will call the triple $(\frakp,\frakq, M)$ \emph{leading} if \[p^{|\frakp|+|\frakq|-|P|}\big|S_0(A_\frakp \overset{M}{\times}A_\frakq, (\bm{b},\bm{b}))\cap[n]^{|\frakp|+|\frakq|-|P|}\big| = \Omega(\mu_2),\] and hence \[\mu_2 \sim \sum_{\substack{(\frakp,\frakq,M)\\ \text{leading}}}p^{|\frakp|+|\frakq|-|\dom(M)|}\big|S_0(A_\frakp \overset{M}{\times}A_\frakq, (\bm{b},\bm{b}))\cap[n]^{|\frakp|+|\frakq|-|\dom(M)|}\big|.\]
Let us first make an observation that will be helpful later.
Suppose $|\frakp|\geq |\frakq|$, and let $M\colon P\to Q$ be a bijection between some index sets $P\subset [|\frakp|]$ and $Q\subset [|\frakq|]$.
Then by Lemma~\ref{lemma:compound_matrix_LB} there exists a $P'\supset P$ such that using $np\to\infty$ and $p\to 0$ we see
\begin{equation}\label{eq:if_case3_leadingnonmixed}
\begin{split}
\mu_2 &= \Omega\left(p^{2|\frakp|-|P'|}\big|S_0\big(A_\frakp \overset{\id_{P'}}{\times} A_\frakp, (\bm{b},\bm{b})\big)\cap[n]^{2|\frakp|-|P'|}\big|\right)\\
&= \Omega\left(p^{2|\frakp|-|P|}\big|S_0\big(A_\frakp \overset{\id_{P'}}{\times} A_\frakp, (\bm{b},\bm{b})\big)\cap[n]^{2|\frakp|-|P'|}\big|\right)\\
&= \Omega\left(n^{2|\frakp|-2\rk(A)-(|P|-r_P(A_\frakp))}p^{2|\frakp|-|P|}\right)\\
&= \Omega\left(n^{|\frakp|+|\frakq|-2\rk(A)-(|P|-r_P(A_\frakp))}p^{|\frakp|+|\frakq|-|P|}\right)\\
&= \Omega\left(p^{|\frakp|+|\frakq|-|P|}\big|S_0\big(A_\frakp \overset{M}{\times} A_\frakq, (\bm{b},\bm{b})\big)\cap[n]^{2|\frakp|-|P|}\big|\right).
\end{split}
\end{equation}
Here the last line followed from the fact that $\rk(A_\frakp \overset{M}{\times} A_\frakq)\geq \rk(A_\frakq)+\rk(A_\frakp^{\widebar{P}})=2\rk(A)-r_P(A_\frakp)$.
Hence, $(\frakp,\frakq,M)$ can be a leading overlap only if
\begin{enumerate}[label=\roman*)]
\item $|\frakp|=|\frakq|$,
\item $r_P(A_\frakp)=r_Q(A_\frakq)$,
\item $\rk(A_\frakp \overset{M}{\times} A_\frakq) = 2\rk(A)-r_P(A_\frakp)$,
\item For any matrix $C\in\left\{A_\frakp \overset{\id_{P}}{\times} A_\frakp,A_\frakq \overset{\id_{Q}}{\times} A_\frakq,A_\frakp \overset{M}{\times} A_\frakq\right\}$ it holds that \[|S_0(C\cap[n]^{2|\frakp|-|P|})|=\Theta(n^{2|\frakp|-(|P|-r_P(A_\frakp))}),\] and
\item $(\frakp,\frakp,\id_P)$ and $(\frakq,\frakq,\id_Q)$ are leading overlaps,
\end{enumerate}
since otherwise some of the $\Omega$ terms in Equation~\eqref{eq:if_case3_leadingnonmixed} would turn into '$\omega$' terms.
Note that these things in particular imply that there cannot exist any $P'\supsetneqq P$ satisfying \[|S_0(A_\frakp\overset{\id_{P'}}{\times}A_\frakp,\bm{b}^2)\cap[n]^{2|\frakp|-|P'|})|=\Omega(n^{2|\frakp|-(|P|-r_P(A_\frakp))}),\] since $p\to 0$ would then imply that $(\frakp,\frakp,\id_P)$ is not a leading triple.
By the proof of Lemma~\ref{lemma:compound_matrix_LB} this means that there does not exist any index $i\in[|\frakp|]$ such that $x_i=0$ for every solution $\bm{x}\in S(A_\frakp,\bm{0})$.
Before continuing, let us quickly try to understand the preceding statements.
Essentially, one would naively hope that all leading triples are of the form $(\frakp,\frakp,\id_P)$ for some $P\subset[|\frakp|]$, but this is too optimistic because it ignores the inherent symmetry of some systems of linear equations.
For instance, it is clear that for $A=\begin{pmatrix}1 &1 &1 &-1\end{pmatrix}$, the partitions $\frakp = (\{1,2\},\{3\},\{4\})$ and $\frakq = (\{1,3\},\{2\},\{4\})$ are technically different, but essentially the same in the sense that there exists a bijection between $S_0(A_\frakp,\bm{b})$ and $S_0(A_\frakq,\bm{b})$.

We can now continue with the actual proof by defining $\frakS''_k\subset\frakS'_k$ to be the set of $k$-tuples $\chi=(\bm{x}_1,\dots,\bm{x}_k)$ such that whenever $1\leq i<j\leq k$ and $|\{\bm{x}_i\}\cap\{\bm{x}_j\}|=s>0$, it holds that $(\frakp(\bm{x}_i),\frakp(\bm{x}_j), M_{i,j})$ is a leading triple.
Here $M_{i,j}$ is the bijection defining the incidences between the distinct components of $\bm{x}_i$ and $\bm{x}_j$.
We will show that
\[\sum_{\chi\in\frakS'_k\setminus\frakS''_k}\Exp\bigg(\prod_{\bm{x}\in\chi}\Ind_{\bm{x}}\bigg) = o(\mu_2^{k/2}).\]
To see this, we will again define a map $\pi\colon \frakS'_k\setminus\frakS''_k\to \frakS'_{k-2}$ in the following way.
Among all $1\leq i<j\leq k$ such that $(\frakp(\bm{x}_i),\frakp(\bm{x}_j), M_{i,j})$ is not a leading triple, let $\{i_\chi,j_\chi\}$ be the set minimizing $\min\{i,j\}$.
Then we can define $\pi(\chi)=\chi^{[k]\setminus\{i_\chi,j_\chi\}}$ and see that
\begin{align*}
\sum_{\chi\in\frakS'_k\setminus\frakS''_k}\Exp\bigg(\prod_{\bm{x}\in\chi}\Ind_{\bm{x}}\bigg) &= \sum_{\chi\in\frakS'_k\setminus\frakS''_k}\Exp\bigg(\prod_{\bm{x}\in\pi(\chi)}\Ind_{\bm{x}}\bigg)p^{|\frakp(\bm{x}_{i_\chi})|+|\frakp(\bm{x}_{j_\chi})|-|\dom(M_{i_\chi,j_\chi})|}\\
&\leq \mu_{k-2}\sum_{\substack{(\frakp,\frakq,M)\\ \text{not leading}}} p^{|\frakp|+|\frakq|-|\dom(M)|}\big|S_0(A_\frakp \overset{M}{\times}A_\frakq, (\bm{b},\bm{b}))\cap[n]^{|\frakp|+|\frakq|-|\dom(M)|}\big|\\
&= o(\mu_2^{k/2}),
\end{align*}
the last line following from the induction hypothesis $\mu_{k-2}=O(\mu_2^{(k-2)/2})$ and the definition of leading triples, noting that there are only $O(1)$ choices for $\frakp$, $\frakq$ and $M$.
Let us enumerate the set of leading triples by $(\frakp_1,\frakq_1,M_1),\dots,(\frakp_u,\frakq_u,M_u)$, and for any choice of $0\leq i_1,\dots,i_u\leq k/2$ such that $\sum i_j = k/2$, define the matrix $A(i_1,\dots,i_u)$ by
\begin{align*}
A(i_1,\dots,i_u) &= \underbrace{(A_{\frakp_1}\overset{M_1}{\times}A_{\frakq_1})\overset{\bm{.}}{\times}\cdots\overset{\bm{.}}{\times}(A_{\frakp_1}\overset{M_1}{\times}A_{\frakq_1})}_{\text{$i_1$ times}} \overset{\bm{.}}{\times} \cdots \overset{\bm{.}}{\times}\underbrace{(A_{\frakp_u}\overset{M_u}{\times}A_{\frakq_u})\overset{\bm{.}}{\times}\cdots\overset{\bm{.}}{\times}(A_{\frakp_u}\overset{M_u}{\times}A_{\frakq_u})}_{\text{$i_u$ times}},
\end{align*}
and write $\ell_j$ for the expression $|\frakp_j|+|\frakq_j|-|\dom(M_j)|$.
We also need the following result.

\begin{claim}\label{claim:almostAllDisjoint}
For any integers $0\leq i_1,\dots,i_u\leq k/2$ satisfying $\sum i_j = k/2$ it holds that \[|S_0(A(i_1,\dots,i_u),\bm{b}^k)\cap [n]^{i_1\ell_1 + \dots + i_u\ell_u}| \sim \prod_{j\in[u]} |S_0(A_{\frakp_j}\overset{M_j}{\times}A_{\frakq_j}, \bm{b}^2)\cap[n]^{\ell_j}|^{i_j}.\]
\end{claim}
\begin{proof}[Proof of the claim:]
Let $(\frakp,\frakq,M)$ be a leading triple with $P:=\dom(M)$.
We will first show that $A_\frakp^{\widebar{P}}$ does not contain any column $c_i$ such that \[\rk(A_\frakp^{\widebar{P}\smallsetminus\{i\}}) = \rk(A_\frakp^{\widebar{P}})-1.\]
But this follows since for any such hypothetical column, every solution $\bm{x}\in S(A_\frakp^{\widebar{P}},\bm{0})$ would satisfy $x_i=0$, and we have seen before that this cannot happen for any leading triple.
We are now able to prove the claim.
Consider any $k/2$-tuple $\chi$ of solutions in \(\bigtimes_{j=1}^u \bigl(S_0(A_{\frakp_j}\overset{M_j}{\times}A_{\frakq_j}, \bm{b}^2)\cap[n]^{\ell_j}\bigr)^{i_j},\) of which there are $\prod_{j\in[u]} |S_0(A_{\frakp_j}\overset{M_j}{\times}A_{\frakq_j}, \bm{b}^2)\cap[n]^{\ell_j}|^{i_j}$ many.
When these solutions are all pairwise disjoint, they in fact define a solution in $S_0(A(i_1,\dots,i_u),\bm{b}^k)\cap[n]^{\ell_1 i_1+\dots+\ell_u i_u}$.
Otherwise, there exist indices $s,t\in[u]$ and a bijection $M\colon P\to Q$ with $\emptyset \neq P \subset [\ell_s]$ and $\emptyset \neq Q\subset [\ell_t]$ such that some pair $\bm{x},\bm{y}\in\chi$ defines a solution in \(S_0\bigl((A_{\frakp_s}\overset{M_s}{\times} A_{\frakq_s})\overset{M}{\times}(A_{\frakp_t}\overset{M_t}{\times} A_{\frakq_t}), \bm{b}^4\bigr)\cap[n]^{\ell_s+\ell_t-|P|}\).
But if we define $Q_1 = Q\cap [|\frakp_t|]$ and $Q_2 = Q\cap [\ell_t]\setminus [|\frakp_t|]$, we see that
\begin{align*}
\rk\bigl((A_{\frakp_s}\overset{M_s}{\times} A_{\frakq_s})\overset{M}{\times}(A_{\frakp_t}\overset{M_t}{\times} A_{\frakq_t})\bigr) &\geq \rk\bigl(A_{\frakp_s}\overset{M_s}{\times}A_{\frakq_s}\bigr) + \rk\bigl((A_{\frakp_t}\overset{M_t}{\times}A_{\frakq_t})^{\widebar{Q}}\bigr)\\
&\geq \rk\bigl(A_{\frakp_s}\overset{M_s}{\times}A_{\frakq_s}\bigr) + \rk(A_{\frakp_t}^{\widebar{Q}_1}) + \rk(A_{\frakq_t}^{\widebar{\im(M_t)}\setminus Q_2})\\
&\geq \rk\bigl(A_{\frakp_s}\overset{M_s}{\times}A_{\frakq_s}\bigr) + \rk(A_{\frakp_t})-|Q_1|+1+\rk{A_{\frakq_t}^{\widebar{\im(M_t)}}}-|Q_2|+1\\
&= \rk\bigl(A_{\frakp_s}\overset{M_s}{\times}A_{\frakq_s}\bigr) + 2\rk(A)-|Q|+2-r_{\im(M_t)}(A_{\frakq_t}).
\end{align*}
Here, for the third inequality we have used the previously established fact that removing any one column from $A_{\frakq_t}^{\widebar{\im(M_t)}}$ will not reduce the rank.
The same clearly holds for $A_{\frakp_t}$ since this matrix is positive.

We are almost done, since now we can use the upper bound from Lemma~\ref{lemma:proper_sols_count} and see that
\begin{align*}
&\bigl|S_0\bigl((A_{\frakp_s}\overset{M_s}{\times} A_{\frakq_s})\overset{M}{\times}(A_{\frakp_t}\overset{M_t}{\times} A_{\frakq_t}), \bm{b}^4\bigr)\cap[n]^{\ell_s+\ell_t-|P|}\bigr|\\
\leq{}& n^{\ell_s+\ell_t-|Q|-\rk\bigl((A_{\frakp_s}\overset{M_s}{\times} A_{\frakq_s})\overset{M}{\times}(A_{\frakp_t}\overset{M_t}{\times} A_{\frakq_t})\bigr)}\\
\leq{}& n^{\ell_s+\ell_t-\rk(A_{\frakp_s}\overset{M_s}{\times}A_{\frakq_s})-2\rk(A)+r_{\im(M_t)}(A_{\frakq_t})}\\
={}& o\bigl(|S_0(A_{\frakp_s}\overset{M_s}{\times}A_{\frakq_s},\bm{b}^2)\cap[n]^{\ell_s}|\cdot |S_0(A_{\frakp_t}\overset{M_t}{\times}A_{\frakq_t},\bm{b}^2)\cap[n]^{\ell_t}|\bigr),
\end{align*}
the last line following from the observations made from~\eqref{eq:if_case3_leadingnonmixed}.
Since there are only $O(1)$ many choices for $s$, $t$, and $M$, this proves the claim.
\end{proof}

Using the claim, we can now conclude
\begin{align*}
\mu_k &\sim \frac{k!}{2^{k/2}}\sum_{\substack{i_1,\dots,i_u\\ i_1+\dots+i_u = k/2}} \frac{1}{i_1!\cdots i_u!} p^{i_1\ell_1+\dots+i_u\ell_u} |S_0(A(i_1,\dots,i_u),\bm{b}^k)\cap [n]^{i_1\ell_1 + \dots + i_u\ell_u}|\\
&\sim \frac{k!}{(k/2)! 2^{k/2}} \sum_{\substack{i_1,\dots,i_u\\ i_1+\dots+i_u = k/2}} \binom{k/2}{i_1,\dots,i_u} \prod_{j\in[u]} \bigl(p^{\ell_j}|S_0(A_{\frakp_j}\overset{M_j}{\times}A_{\frakq_j}, \bm{b}^2)\cap[n]^{\ell_j}|\bigr)^{i_j}\\
&\sim \frac{k!}{(k/2)!\, 2^{k/2}}\mu_2^{k/2},
\end{align*}
which is what we wanted to prove.
\hfill\qedsymbol

\section{Concluding remarks}\label{sec:conclusion}

In this paper we have established sufficient conditions for the choice of $p$ in order to guarantee normal limiting distributions for the number of solutions to linear systems of equations of the form $A\bm{x} = \bm{b}$.
Specifically, we showed that $n(1-p)\rightarrow \infty$ and $np^{c(A_{\mathfrak{p}})}\rightarrow \infty$ (for all partitions under consideration) sufficed, and we used them in their full strength in Cases~2 and 3, respectively of the proof of Theorem~\ref{thm:distribution}. 
Both of these conditions are analogous to those that appear in Ruci\'nski's proof of normality for the number of copies of a given subgraph $H$ in the binomial random model $G(n,p)$.
Note that when comparing to the graph setting, the elements of $[n]$ in systems take on the function of both vertices and edges, and hence the analogue of our  $n(1-p)\rightarrow \infty$ requirement is to ask that $n^2(1-p)$, the expected number of edges, is unbounded.

In fact, in~\cite{Rucinski1988} Ruci\'nski showed that those conditions were necessary as well as sufficient. 
In our context we can say something similar regarding the condition that $n(1-p)$ is unbounded.
Observe that the expression $n(1-p)$ can be interpreted as the expected number of elements that are \emph{not} chosen in $[n]_p$, hence if $n(1-p)\not \rightarrow \infty$, then $[n]_p$ is typically all the interval $[n]$ with the exception of a bounded number of elements. 
Then it is easy to show that under this condition, $X_n$ is strongly concentrated around its mean value, concluding that $\tilde{X}_n \xrightarrow{d} 0$, and hence, we de not have a normal limiting distribution for the number of solutions. 

However, the argument used by Ruci\'nski in order to study the second condition requires a delicate study of moments of order $4$ and $6$, which we are unable to adapt. 
Roughly speaking, in our algebraic setting the structure of solutions is more complicated compared with the graph setting, as solutions with repeated components can be valid ones (the analogy would be to consider also subgraphs of the fixed graph $H$ that we want to count the corresponding number of subcopies). 
Hence, an open question arising from our work is to obtain an \emph{only if} statement of Theorem~\ref{thm:distribution}.
As an intermediate step, one could investigate the two main corollaries of our meta theorem, Theorems~\ref{thm:distProper} and~\ref{thm:distHomStrictBal}, and in fact, the following arguments suggest that the sufficient conditions on $p$ stated in these theorems are also necessary.

Let us start with the setting of Theorem~\ref{thm:distHomStrictBal}, that of non-trivial solutions in strictly balanced homogeneous systems of linear equations.
Here, as discussed before, Ru\'e, Spiegel and Zumalac\'{a}rregui in~\cite{RSZ2018} already established a threshold result showing that when $np^{c(A)}=o(1)$, asymptotically almost surely $S_1(A,\mathbf{0})\cap[n]^m_p$ is empty.
Using a concentration argument similar to the one above when discussing the necessity of $n(1-p)\to\infty$, we see that in this case $|S_1(A,\mathbf{0})\cap[n]^m_p|$ will converge in distribution to the constant $0$ distribution.
In addition to this threshold results, the authors also studied the distribution at the threshold, that is, the case $np^{c(A)}\to a>0$ where $a$ is a constant.
Their results show that here, the random variable $|S_1(A,\mathbf{0})\cap[n]^m_p|$ converges in distribution to a Poisson.
Putting all of this together, we see that this establishes the only if direction for Theorem~\ref{thm:distHomStrictBal}.

Let us now turn to the setting of Theorem~\ref{thm:distProper}, that of proper solutions, where there are no repeated variables by definition.
The argument that $np^{c(A)}=o(1)$ implies that $\tilde{X}_n \xrightarrow{d} 0$ is the same as before, so what remains is to understand the case when $np^{c(A)}$ tends to some positive constant.
When $A$ is strictly balanced, one can use the same arguments that were used by Ru\'e, Spiegel and Zumalac\'{a}rregui in the proof of the strictly balanced homogeneous case for the distribution of non-trivial solutions to see that one will also have a Poisson distribution in the setting of Theorem~\ref{thm:distProper}.
When $A$ is not strictly balanced, an analysis as was performed by Ruci\'nski in the subgraph setting is needed, but since we are now in the situation that all variables are pairwise distinct, the same result should follow, which would establish the necessity of $np^{c(A)}\to\infty$.

To conclude, let us mention that once one has proved a normal limiting distribution, a natural next question is the study local limit theorems as well as anti-concentration results and tail estimates in a general context.
This has been a very active trend of research in the last years, see for instance~\cite{Warnke2017,BeSaSa2021}.

\section*{Acknowledgments}
{\setlength\parindent{0pt}
We thank Oriol Serra for a preliminary reading of this manuscript, as well as suggestions and improvements in the presentation of the results. We also thanks the anonymous referees of the paper, whose suggestions had improved the final presentation of the paper.

This research was conducted while Maximilian W\"otzel was a member of the Barcelona Graduate School of Mathematics (BGSMath) as well as the Universitat Polit\`ecnica de Catalunya.
Maximilian W\"otzel acknowledges financial support from the Fondo Social Europeo and the Agencia Estatal de Investigaci\'on through the FPI grant number MDM-2014-0445-16-2 and the Spanish Ministry of Economy and Competitiveness, through the Mar\'ia de Maeztu Programme for Units of Excellence in R\&D (MDM-2014-0445), through the project MTM2017-82166-P, as well as from the Dutch Science Council (NWO) through the grant number OCENW.M20.009.}

Juanjo Ru\'e acknowledges financial support from Spanish State Research Agency through projects MTM2017-82166-P, PID2020-113082GB-I00,
the Severo Ochoa and María de Maeztu Program for Centers and Units of Excellence in R\&D (CEX2020-001084-M), and the Marie Curie RISE research network 'RandNet' MSCA-RISE-2020-101007705.

\bibliographystyle{amsplain}
\bibliography{refs}

\providecommand{\MR}[1]{}
\providecommand{\bysame}{\leavevmode\hbox to3em{\hrulefill}\thinspace}
\providecommand{\MR}{\relax\ifhmode\unskip\space\fi MR }
\providecommand{\MRhref}[2]{%
  \href{http://www.ams.org/mathscinet-getitem?mr=#1}{#2}
}
\providecommand{\href}[2]{#2}
\begin{thebibliography}{10}

\bibitem{BaMoSa15}
J\'{o}zsef Balogh, Robert Morris, and Wojciech Samotij, \emph{Independent sets
  in hypergraphs}, J. Amer. Math. Soc. \textbf{28} (2015), no.~3, 669--709.
  \MR{3327533}

\bibitem{BaAnKoLi2019}
Yacine {Barhoumi-Andr\'eani}, Christoph {Koch}, and Hong {Liu},
  \emph{{Bivariate fluctuations for the number of arithmetic progressions in
  random sets}}, {Electron. J. Probab.} \textbf{24} (2019), 32 (English), Id/No
  145.

\bibitem{BeSaSa2021}
Ross {Berkowitz}, Ashwin {Sah}, and Mehtaab {Sawhney}, \emph{{Number of
  arithmetic progressions in dense random subsets of
  \(\mathbb{Z}/n\mathbb{Z}\)}}, {Isr. J. Math.} \textbf{244} (2021), no.~2,
  589--620 (English).

\bibitem{CoGo16}
D.~{Conlon} and W.~T. {Gowers}, \emph{{Combinatorial theorems in sparse random
  sets}}, {Ann. Math. (2)} \textbf{184} (2016), no.~2, 367--454 (English).

\bibitem{FGR88}
Peter {Frankl}, Ronald~L. {Graham}, and Vojt\v{e}ch {R\"odl},
  \emph{{Quantitative theorems for regular systems of equations}}, {J. Comb.
  Theory, Ser. A} \textbf{47} (1988), no.~2, 246--261 (English).

\bibitem{HaStTr2019}
Robert {Hancock}, Katherine {Staden}, and Andrew {Treglown}, \emph{{Independent
  sets in hypergraphs and Ramsey properties of graphs and the integers}}, {SIAM
  J. Discrete Math.} \textbf{33} (2019), no.~1, 153--188 (English).

\bibitem{KoLuRo96}
Yoshiharu {Kohayakawa}, Tomasz {{\L}uczak}, and Vojt\v{e}ch {R\"odl},
  \emph{{Arithmetic progressions of length three in subsets of a random set}},
  {Acta Arith.} \textbf{75} (1996), no.~2, 133--163 (English).

\bibitem{Roth53}
Klaus~F. {Roth}, \emph{{On certain sets of integers}}, {J. Lond. Math. Soc.}
  \textbf{28} (1953), 104--109 (English).

\bibitem{Rucinski1988}
Andrzej Ruci\'{n}ski, \emph{When are small subgraphs of a random graph normally
  distributed?}, Probab. Theory Related Fields \textbf{78} (1988), no.~1,
  1--10. \MR{940863}

\bibitem{RuSeVe17}
Juanjo Ru\'{e}, Oriol Serra, and Lluis Vena, \emph{Counting configuration-free
  sets in groups}, European J. Combin. \textbf{66} (2017), 281--307.

\bibitem{RSZ2018}
Juanjo Ru\'{e}, Christoph Spiegel, and Ana Zumalac\'{a}rregui, \emph{Threshold
  functions and {P}oisson convergence for systems of equations in random sets},
  Math. Z. \textbf{288} (2018), no.~1-2, 333--360. \MR{3774416}

\bibitem{Ruzsa1993}
Imre~Z. Ruzsa, \emph{Solving a linear equation in a set of integers. {I}}, Acta
  Arith. \textbf{65} (1993), no.~3, 259--282. \MR{1254961}

\bibitem{SaTh15}
David Saxton and Andrew Thomason, \emph{Hypergraph containers}, Invent. Math.
  \textbf{201} (2015), no.~3, 925--992.

\bibitem{Scha16}
Mathias Schacht, \emph{Extremal results for random discrete structures}, Ann.
  of Math. (2) \textbf{184} (2016), no.~2, 333--365.

\bibitem{Spiegel2017}
Christoph {Spiegel}, \emph{{A note on sparse supersaturation and extremal
  results for linear homogeneous systems}}, {Electron. J. Comb.} \textbf{24}
  (2017), no.~3, research paper p3.38, 19 (English).

\bibitem{Spiegel2020}
Christoph Spiegel, \emph{Additive structures and randomness in combinatorics},
  Ph.D. thesis, Universitat Politècnica de Catalunya, 2020, Available at
  \url{http://hdl.handle.net/2117/328203}.

\bibitem{Szemeredi75}
Endre Szemer\'{e}di, \emph{On sets of integers containing no {$k$} elements in
  arithmetic progression}, Acta Arith. \textbf{27} (1975), 199--245.

\bibitem{Warnke2017}
Lutz Warnke, \emph{Upper tails for arithmetic progressions in random subsets},
  Israel J. Math. \textbf{221} (2017), no.~1, 317--365.

\end{thebibliography}




\end{document}